\documentclass[11pt]{amsart}

\usepackage[T1]{fontenc}
\usepackage[a4paper,top=4cm,bottom=4cm,inner=3cm,outer=3cm]{geometry}
\usepackage[pdftex]{graphicx,graphics} 
\usepackage{amsmath,amssymb,amsfonts,color,hyperref} 
 \usepackage{epsfig,array}

\usepackage{graphicx,psfrag,mathrsfs}
\usepackage{tikz}  
\usetikzlibrary{shapes.geometric, arrows,snakes}

\hypersetup{
    linktoc=page,
   linkcolor=red,          
  citecolor=blue,        
    filecolor=blue,      
   urlcolor=cyan,
    colorlinks=true           
}

\usepackage[refpage]{nomencl}

\newcommand{\lan}{{\mathcal L}}
\newcommand\var{\text{var}}
\newcommand\vw{\nu_{\omega}}

\newcommand\aw{\alpha_{\omega}}

\newcommand\vl{\underline{v}}
\newcommand\wl{\underline{w}}
\newcommand\dl{{\underline{\dim}_{\rm loc}}}
\newcommand\du{{\overline{\dim}_{\rm loc}}}
\newcommand\dm{{{\dim}_{\rm loc}}}
\newcommand\El{\underline{E}}
\newcommand\supp{\mathrm{supp}}

\RequirePackage{yhmath} 
\renewcommand\widering[1]{\ring{#1}}

\newcommand\N {{\Bbb N}}

\newtheorem{theorem}{Theorem}[section]
\newtheorem{corollary}{Corollary}

\newtheorem{lemma}[theorem]{Lemma}
\newtheorem{proposition}{Proposition}

\theoremstyle{definition}
\newtheorem{definition}[theorem]{Definition}
\newtheorem{remark}{Remark}

\newtheorem{fact}{Fact}

\title[Inverse measures of random weak Gibbs measures] 
{Multifractal formalism for inverse measures of random weak Gibbs measures}

\author[Zhihui YUAN]{}


\email{yzhh@hust.edu.cn}
\begin{document}
\maketitle

\centerline{\scshape Zhihui Yuan}
\medskip

\begin{abstract}
	Any Borel probability measure supported on a Cantor set of zero Lebesgue measure on the real line possesses a discrete inverse measure. We study the validity of the multifractal formalism for the inverse measures of random weak Gibbs measures supported on  the attractor associated with some $C^1$ random dynamics encoded by a random subshift of finite type, and expanding in the mean. The study requires, in particular, to develop in this context of random dynamics a suitable extension of the results known for heterogeneous ubiquity associated with deterministic Gibbs measures.
\end{abstract}

\section{Introduction} This paper investigates the validity of the multifractal formalism for discrete measures naturally arising in some random dynamical systems, namely the inverse measures of random weak Gibbs measures supported on random dynamical attractors in the line, to be defined below.

Let us recall the basic framework of multifractal formalism. Given a positive, finite, and compactly supported  Borel measure $\mu$ on $\mathbb{R}$, whose topological support is denote by $\supp(\mu)$, one defines its $L^q$-spectrum  $\tau_{\mu}:\mathbb{R}\rightarrow\mathbb{R}\cup\{-\infty\}$ by
\begin{equation}
\tau_{\mu}(q)=\liminf_{r\rightarrow 0}\frac{\log \sup\{\sum_{i}(\mu(B_{i}))^q\}}{\log (r)},
\end{equation}
where the supremum is taken over all families of disjoint closed balls $B_i$ of radius $r$ with centers in $\text{supp}(\mu)$.
One also defines for any $x\in \mathrm{supp}(\mu)$ the lower local dimension and the upper local dimension of $\mu$ at $x$ by
$$\dl(\mu,x)=\liminf_{r\to 0^+}\frac{\log \mu(B(x,r))}{\log r} \text{ and } \du(\mu,x)=\limsup_{r\to 0^+}\frac{\log \mu(B(x,r))}{\log r},$$
and the local dimension of $\mu$ at $x$ by  $\dm(\mu,x)=\dl(\mu,x)=\du(\mu,x)$ if the last equality holds. Then, one considers the level sets associated with these quantities, i.e., for $d\in \mathbb{R}$, the sets
\begin{align*}
\El(\mu,d)&=\{x\in \supp(\mu) :\dl(\mu,x)=d\},\\
\overline{E}(\mu,d)&=\{x\in \supp(\mu) :\du(\mu,x)=d\},\\
E(\mu,d)&=\El(\mu,d)\cap\overline{E}(\mu,d).\\
\end{align*}
One also defines the lower Hausdorff dimension of $\mu$:
$$\dim_H(\mu)=\sup\{s: \dl(\mu,x)\geq s\text{ for }\mu\text{-almost every }x\in \supp(\mu) \}.$$
An equivalent definition is (see~\cite[chapter 10]{Falconer1}) $\dim_H(\mu)=\inf\{\dim_H E: \,E\text{ Borel set},\, \mu(E)>0\}.
$
One always has, for all $d\in\mathbb{R}$ (see \cite{Olsen1,LN})
\begin{align}
\nonumber\dim_H E(\mu,d)&\le \min (\dim_H \El(\mu,d), \dim_H \overline E(\mu,d))\\
\label{MFI}&\le  \max (\dim_H \El(\mu,d), \dim_H \overline E(\mu,d))\le \tau_\mu^*(d),
\end{align}
where $\dim_H$ stands for the Hausdorff dimension, a negative dimension means that the set is empty, and we recall that the Legendre transform $f^*$ of any function $f:\mathbb{R}\to \mathbb{R}\cup\{-\infty\}$ with non-empty domain is defined on $\mathbb{R}$  by
$$f^*(d)=\inf_{q\in \mathbb{R}}\{d q-f(q)\}\in \mathbb{R}\cup\{-\infty\}.$$
The Hausdorff spectrum of $\mu$ is defined by
$$
d\in\mathbb{R} \mapsto \dim_H E(\mu,d),
$$
while the lower and upper Hausdorff spectra are defined similarly with the sets $\El(\mu,d)$ and $\overline E(\mu,d)$ respectively, in place of $E(\mu,d)$.

\begin{definition}
	One says that the multifractal formalism holds on a subset $I$ of $\mathbb{R}$ if $
	\dim_H \El(\mu,d)=\tau_\mu^*(d)$ for all  $d\in I$ and that it holds strongly on $I$ if
	$
	\dim_H E(\mu,d)=\tau_\mu^*(d)$ for all  $d\in I
	$. If $I=\mathbb{R}$, one simply says that the formalism holds.
\end{definition}
The study of the validity of the multifractal formalism for discrete measures started in \cite{MR1997-1}, and was developed further in~\cite{RM1998}. In these papers, the authors study the relations between the multifractal behavior of a Borel probability measure  supported  on $[0,1]$ and its inverse measure defined as follows:
\begin{definition}
	Let $\mu$ be a Borel probability measure supported  on $[0,1]$, and let $F_{\mu}$ be its distribution function, i.e. $F_{\mu}(t)=\mu([0,t])$. The  inverse measure $\nu$ of $\mu$ is the unique Borel probability measure on $[0,1]$ such that $\text{for all } x\in [0,1], F_{\nu}(x)=\sup\{t\in[0,1];F_{\mu}(t)\leq x\}.$
\end{definition}
The authors of \cite{RM1998} use  $\lim_{I\to\{x\}} \frac{\log(\mu(I))}{\log(|I|)}$ as a definition of  the local dimension, where~$I$ is a non trivial interval containing~$x$.  With this definition,  they observe that for a Gibbs measure on a cookie-cutter set, while it is well known  that the strong multifractal formalism holds, they can establish its failure on a non trivial interval for the discrete inverse measure of such a measure (they obtained the same type of failure for discrete in-homogeneous self-similar measures, see also \cite{OlSni}). Later,  the validity of the multifractal formalism as defined above was obtained in \cite{BS2009}, where the authors used the so-called heterogeneous, or conditioned, ubiquity theory, which combines ergodic theory and metric approximation theory, and was developed in~\cite{BS2007}. This tool makes it possible to study  a broad class of multifractal discrete measures~\cite{BS2004,BS2008}.

It is worth mentioning that the multifractal analysis of discrete measures via the Hausdorff dimensions of the level sets $\underline E(\mu,d)$, and with no consideration of multifractal formalism, started with homogeneous sums of Dirac masses \cite{AB1990,Jaffard1996,Jaffard1999,Falconer2000}, in particular the derivative of L\'evy subordinators \cite{Jaffard1999}, and that originally heterogeneous ubiquity was elaborate with the multifractal analysis of L\'evy processes in multifractal time as a target~\cite{BSadv}.

In \cite{Yuan2016}, we consider, on a base probability space $(\Omega,\mathcal F,\mathbb{P},\sigma)$, random weak Gibbs measures $\{\mu_{\omega}\}_{\omega\in\Omega}$ on some class of attractors $\{X_{\omega}\}_{\omega\in\Omega}$ included in $[0,1]$ and associated with $C^1$ random dynamics conjugate (up to countably many points), or semi-conjugate to a random subshift of finite type. We provide a study of the multifractal nature of these measures, including the validity of the strong multifractal formalism, the calculation of Hausdorff and packing dimensions of the so-called level sets of divergent points,  and a $0$-$\infty$ law for the Hausdorff and packing measures of the level sets of the local dimension.

In the present work, we study the multifractal nature of the discrete measures obtained as the inverse measures of the random weak Gibbs measures $\{\mu_{\omega}\}_{\omega\in\Omega}$, when the attractors have zero Lebesgue measure.
The precise definitions of these objects and our main result, theorem~\ref{multifractal inverse}, are exposed in the next section. Let us just mention that  the randomness and  the fact that we work on a subshift rather than  a fullshift are two sources of serious complications with respect to the study achieved for inverse measures of deterministic Gibbs measures in \cite{BS2009}.  In particular, the fundamental geometric tool provided by \cite{BS2007} must be revisited.

\section{Setting and main result}\label{section:Setting and main results}
We first need to expose basic facts from random dynamical systems and thermodynamic formalism.

\subsection{Random subshift and random weak Gibbs measures}\label{subsection: Random subshift, relativized entropy, topological pressure  and weak Gibbs measures}
\medskip
{\bf Random subshift.} 
Denote by $\Sigma$ the symbolic space $({\mathbb Z^+})^\N$, and endow it with the standard ultrametric distance: for any $\underline u=u_0u_1\cdots$ and $\vl=v_0v_1\cdots$ in $\Sigma$, $d(u,v)=e^{-\inf\{n\in\mathbb N: \ u_n\neq v_n\}}$, with the convention $\inf (\emptyset)=+\infty$. Let $(\Omega,\mathcal{F},\mathbb{P})$ be a complete probability space and $\sigma$  a $\mathbb{P}$-preserving  ergodic map. The product space $\Omega\times \Sigma$ is endowed with the $\sigma$-field $\mathcal F\otimes \mathcal B(\Sigma)$, where $\mathcal B(\Sigma)$ stands for the Borel $\sigma$-field of $\Sigma$.

Let $l$ be a $\mathbb{Z}^+$ valued random variable such that
$
\int\log (l)\, \mathrm{d}\mathbb{P}<\infty$ and $ \mathbb{P}(\{\omega\in\Omega: l(\omega)\geq 2\})>0.
$
Let $A=\{A(\omega)=(A_{r,s}(\omega)):\omega\in\Omega\}$ be  a random transition matrix such that $A(\omega)$ is a $l(\omega)\times l(\sigma\omega)$-matrix with entries 0 or 1. We  suppose that the map $\omega\mapsto A_{r,s} (\omega)$ is measurable for all $(r,s)\in \mathbb{Z}^+\times \mathbb{Z}^+$ and each $A(\omega)$  has at least one non-zero entry in each row and each column.
Let $\Sigma_{\omega}=\{\vl=v_0v_1\cdots;1\leq v_k\leq l(\sigma^{k}(\omega))\text{ and } A_{v_k,v_{k+1}}(\sigma^{k}(\omega))=1\ \text{for all } k\in\mathbb{N} \},$
and $F_{\omega}:\Sigma_{\omega}\rightarrow\Sigma_{\sigma\omega}$ be the left shift $(F_{\omega}\vl)_i=v_{i+1}$ for any $\vl=v_0v_1\cdots\in \Sigma_{\omega}$. Define $\Sigma_{\Omega}=\{(\omega,\vl):\omega\in\Omega, \vl\in\Sigma_\omega\}$
which is endowed with the $\sigma$-field  obtained as the trace of $\mathcal F\otimes \mathcal B(\Sigma)$.
Define the map $F:\Sigma_{\Omega}\to \Sigma_{\Omega}$  as $F((\omega,\vl))=(\sigma\omega,F_{\omega}\vl)$. The corresponding family $\tilde{F}=\{F_{\omega}:\omega\in \Omega\}$ is called a random subshift. We assume that this  random subshift  is topologically mixing, i.e.  there exists a $\mathbb{Z}^+$-valued r.v. $M$ on $(\Omega,\mathcal{F},\mathbb{P})$ such that for $\mathbb{P}$-almost every (a.e.) $\omega$,
$
A(\omega)A(\sigma\omega)\cdots A(\sigma^{M(\omega)-1}\omega)$ is positive.

For each $n\ge 1$, define $\Sigma_{\omega,n}$ as the set of words $v=v_0v_1\cdots v_{n-1}$ of length $n$, i.e. such that
$1\leq v_k\leq l(\sigma^{k}(\omega))$ for all  $0\leq k\leq n-1$
and $A_{v_k,v_{k+1}}(\sigma^{k}(\omega))=1$ for all $ 0\leq k\leq n-2$.
Define $\Sigma_{\omega,*}=\cup_{n\in\mathbb{N}}\Sigma_{\omega,n}$.
For $v=v_0v_1\cdots v_{n-1}\in\Sigma_{\omega,n}$,  we write $|v|$ for the length $n$ of $v$, and we define the cylinder $[v]_{\omega}$ as
$[v]_{\omega}:=\{\wl\in\Sigma_\omega:w_i=v_i\text{ for }i=0,\dots,n-1\}.$

For any $s\in \Sigma_{\omega,1}$, $p\geq M(\omega)$ and $s'\in \Sigma_{\sigma^{p+1}\omega,1}$, there is  at least one word $v(s,s')\in \Sigma_{\sigma\omega,p-1}$ such that $sv(s,s')s'\in \Sigma_{\omega,p+1}$. We fix such a $v(s,s')$ and denote the word $sv(s,s')s'$ by $s\ast s'$. Similarly, for any $w=w_0w_1\cdots w_{n-1}\in\Sigma_{\omega,n}$ and $w'=w'_0w'_1\cdots w'_{m-1}\in\Sigma_{\sigma^{n+p}\omega,m}$ with $n,p,m\in\mathbb{N}$ and $p\geq M(\sigma^{n-1}\omega)$, we  fix $v(w_{n-1},w'_{0})\in \Sigma_{\sigma^n\omega, p}$ (a word depending on  $w_{n-1}$ and $w'_{0}$ only) so that  $w\ast w':=w_0w_1\cdots w_{n-1}v(w_{n-1},w'_{0})w'_0w'_1\cdots w'_{m-1}\in\Sigma_{\omega,n+m+p-1}$.

\medskip

{\bf Random weak Gibbs measures.} We say that a measurable function $\Phi$ on $\Sigma_{\Omega}$ is in $\mathbb{L}^1_{\Sigma_{\Omega}}(\Omega,C(\Sigma))$ if
\begin{enumerate}
	\item \begin{equation}\label{int}
	C_{\Phi}=:\int_{\Omega} \|\Phi(\omega)\|_{\infty}\, \mathrm{d}\mathbb{P}(\omega)<\infty,
	\end{equation}
	where
	$
	\|\Phi(\omega)\|_{\infty}=:\sup_{\underline{v}\in \Sigma_{\omega}}|\Phi(\omega,\vl)|$,
	\item 
	for $\mathbb{P}$-a. e. $\omega$,  $\text{var}_n\Phi(\omega)\to 0$ as $n\to\infty$,
	where
	$$
	\text{var}_n\Phi(\omega) = \sup\{|\Phi(\omega,\vl)-\Phi(\omega,\wl)| : v_i = w_i, \forall i < n\}.
	$$
\end{enumerate}

Now, if $\Phi\in \mathbb{L}^1_{\Sigma_{\Omega}}(\Omega,C(\Sigma))$, due to Kingsman's subadditive ergodic theorem,
$$P(\Phi)=\lim_{n\to\infty}\frac{1}{n}\log \sum_{v\in\Sigma_{\omega,n}}\sup_{\underline v\in[v]_\omega}\exp\left(S_n\Phi(\omega,\underline v)\right )$$
exists for $\mathbb{P}$-a.e.  $\omega$ and does not depend on $\omega$, where $S_n\Phi(\omega,\underline v)=\sum_{i=0}^{n-1}\Phi(F^i(\omega,\underline v))$.  This limit is called {\it topological pressure of $\Phi$}.

Also, with $\Phi$ is associated
the Ruelle-Perron-Frobenius  operator $\lan_{\Phi}^\omega: C^0(\Sigma_{\omega})\to C^0(\Sigma_{\sigma\omega})$ defined as
$$\lan_{\Phi}^\omega h(\vl)=\sum_{F_{\omega}\wl=\vl}\exp(\Phi(\omega,\wl))h(\wl),\ \ \forall\ \vl\in \Sigma_{\sigma\omega}.$$
\begin{proposition}\label{eigen}\cite{Kifer1,MSU}
	Removing from $\Omega$ a set of $\mathbb{P}$-probability 0 if necessary, for all  $\omega\in \Omega$ there exists $\lambda(\omega)=\lambda^{\Phi}(\omega)>0$ and  a probability measure $\widetilde \mu_{\omega}=\widetilde\mu^\Phi_\omega$ on $\Sigma_{\omega}$ such that $(\mathcal{L}_{\Phi}^\omega)^*\widetilde \mu_{\sigma\omega}=\lambda(\omega)\widetilde \mu_{\omega}$.
\end{proposition}

\begin{proposition}\label{eigencon}\cite{Kifer1,MSU,Yuan2016}
	For $\mathbb{P}$-a.e. $\omega\in\Omega$, $$\lim_{n\to\infty}\frac{\log(\prod_{i=0}^{n-1}\lambda^{\Phi}(\sigma^i\omega))}{n}=P(\Phi).$$
\end{proposition}

We call the  family $\{\widetilde \mu^\Phi_{\omega}:\omega\in\Omega\}$ a random weak Gibbs measure on $\{\Sigma_\omega:\omega\in\Omega\}$ associated with $\Phi$.

\subsection{A model of random dynamical attractor}

We present the model of random dynamical attractor in the real line defined and illustrated in \cite{Yuan2016}. It is more general than examples considered until now in the literature dedicated to multifractal analysis of random Gibbs measures on $\mathbb{R}$ \cite{Kifer2,MSU,FS}.

For any $\omega\in\Omega$, let $U_{\omega}^1=[a_{\omega,1},b_{\omega,1}],U_{\omega}^2=[a_{\omega,2},b_{\omega,2}],\cdots U_{\omega}^{s}=[a_{\omega,s},b_{\omega,s}]\cdots$ be closed non trivial intervals with disjoint interiors and $b_{\omega,s}\le a_{\omega,s+1}$. We assume that for each $s\ge 1$, $\omega\mapsto (a_{\omega,s},b_{\omega,s})$ is measurable, as well as  $a_{\omega,1}\geq 0$ and $b_{\omega,l(\omega)}\leq 1$.  Let $f^s_\omega(x)=\frac{x-a_{\omega,s}}{b_{\omega,s}-a_{\omega,s}}$ and consider a measurable  mapping $\omega\mapsto \mathbf{T}_{\omega}^s$ from $(\Omega,\mathcal F)$ to the space of $C^1$ diffeomorphisms of $[0,1]$ endowed with its Borel $\sigma$-field. We consider  the measurable $C^1$ diffeomorphism  $T_{\omega}^s: U_{\omega}^s \rightarrow [0,1]$ by $T_{\omega}^s=\mathbf{T}_{\omega}^s\circ f^s_\omega$. We denote the  inverse of $T_\omega^s$ by $g_{\omega}^s$. We also define
\begin{eqnarray*}
	U_{\omega}^v&=&g_{\omega}^{v_0}\circ g_{\sigma\omega}^{v_1}\circ\cdots\circ g_{\sigma^{n-1}\omega}^{v_{n-1}}([0,1]),\ \forall v=v_0v_1\cdots v_{n-1}\in\Sigma_{\omega,n},\\
	X_{\omega}&=&\bigcap_{n\geq 1}\bigcup_{v\in\Sigma_{\omega,n}}U_{\omega}^v,\\
	X_{\Omega}&=&\{(\omega,x):\omega\in\Omega, x\in X_{\omega}\},
\end{eqnarray*}
and  for all $\omega\in\Omega$,  $ s\ge 1$ and $x\in U_{\omega}^{s}$,
$$
\psi(\omega,s,x)=-\log|(T_{\omega}^{s})'(x)|.
$$

We say that a measurable function $\widetilde \psi$ defined on $\widetilde{U}_{\Omega}=\{(\omega,s,x):\omega\in \Omega,1\leq s\leq l(\omega),x\in U_{\omega}^s\}$ is in $\mathbb{L}^1_{X_{\Omega}}(\Omega,\widetilde{C}([0,1]))$ if
\begin{enumerate}
	\item 
	$\int_{\Omega}\|\widetilde\psi(\omega)\|_{\infty}\mathrm{d}\mathbb{P}(\omega)<\infty$,
	where
	$\|\widetilde\psi(\omega)\|_{\infty}:=\sup_{1\leq s\leq l(\omega)}\sup_{x\in U_{\omega}^{s}}|\widetilde\psi(\omega,s,x)|$,
	
	\item for $\mathbb{P}$-a.e.   $\omega\in\Omega$, $\text{var}(\widetilde\psi,\omega,\varepsilon) \to 0$ as $\varepsilon\to 0$, where
	$$\text{var}(\widetilde\psi,\omega,\varepsilon)=\sup_{1\leq s\leq l(\omega)}\sup_{x,y\in U_{\omega}^s\text{ and } |x-y|\leq \varepsilon }|\widetilde\psi(\omega,s,x)-\widetilde\psi(\omega,s,y)|.$$
\end{enumerate}

We will make the following assumption:

\medskip $\psi\in\mathbb{L}^1_{X_{\Omega}}(\Omega,\widetilde{C}([0,1]))$ and $\psi$ satisfies the  contraction property in the mean
\begin{equation}\label{cpsi}
c_{\psi}:=-\int_{\Omega}\sup_{1\leq s\leq l(\omega)}\sup_{x\in U_{\omega}^{s}}\psi(\omega,s,x)\mathrm{d}\mathbb{P}(\omega)>0.
\end{equation}

Under this assumption, there is $\mathbb{P}$-almost surely a natural projection $\pi_\omega: \Sigma_{\omega}\rightarrow X_{\omega}$ defined as
$$
\pi_\omega(\underline{v})= \lim_{n\rightarrow\infty}g_{\omega}^{v_0}\circ g_{\sigma\omega}^{v_1}\circ\cdots\circ g_{\sigma^{n-1}\omega}^{v_{n-1}}(0).$$
This mapping may not be injective, but any $x\in X_{\omega}$ has at most two preimages in $\Sigma_{\omega}$. Let
$$
\Psi(\omega,\vl)=\psi(\omega,v_0,\pi(\vl))\text{ for }\vl=v_0v_1\cdots\in\Sigma_{\omega}.
$$
By construction we know $ \Psi\in \mathbb{L}^1_{\Sigma_{\Omega}}(\Omega,C(\Sigma))$.
Using a standard approach, it can be easily proven that for $\mathbb{P}$-a.e. $\omega\in \Omega$, the Bowen-Ruelle formula holds, i.e. $\dim_H X_{\omega}=t_0$ where $t_0$ is the unique root of the equation $P(t\Psi)=0$.

\subsection{Main result}

\medskip

Let $\phi\in \mathbb{L}^1_{X_{\Omega}}(\Omega,\widetilde{C}([0,1]))$
and consider the function
$$
\Phi(\omega,\vl)=\phi(\omega,v_0,\pi(\vl)) \quad (\vl=v_0v_1\cdots\in\Sigma_{\omega}).
$$
We have $\Phi\in\mathbb{L}^1_{\Sigma_{\Omega}}(\Omega,C(\Sigma))$. Let $\mu$ be the random weak Gibbs measure on $\{X_\omega:\omega\in \Omega\}$ obtained as $\mu_{\omega}={\pi_{\omega}}_* \widetilde \mu_\omega:= \widetilde \mu_\omega \circ \pi_\omega^{-1}$, where $\widetilde\mu$ is obtained from proposition~\ref{eigen} with respect to $\Phi$. Without changing the random measures $\widetilde\mu_\omega$ and $\mu_\omega$,  we can  assume $P(\Phi)=0$. Then, due to  equation~\eqref{cpsi}, for any $q\in \mathbb{R}$, there exists a unique ${T}(q)\in\mathbb{R}$ such that
$
P(q\Phi-{T}(q)\Psi)=0,
$
and the mapping $T$ is concave and increasing. In \cite{Yuan2016}, we showed that with $\mathbb P$-probability 1, the strong multifractal formalism holds for $\mu_\omega$ with $\tau_{\mu_{\omega}}=T$.

Now we assume
\begin{equation}\label{cphi}
c_{\phi}:=-\int_{\Omega}\sup_{1\leq s\leq l(\omega)}\sup_{x\in U_{\omega}^{s}}(\phi(\omega,s,x))\mathrm{d}\mathbb{P}(\omega)>0,
\end{equation}
and
\begin{equation}\label{lebesgue measure 0}
\quad\mathbb{P}(\{\omega\in\Omega:\text{ The Lebesgue measure of } X_{\omega}\text{ is equal to }0\})=1.
\end{equation}
The first property ensures that for any $q\in \mathbb{R}$, there exists a unique $\mathcal{T}(q)\in\mathbb{R}$ such that
$
P(q\Psi-\mathcal{T}(q)\Phi)=0$; moreover, the mapping $\mathcal T$ is concave and increasing.
The second property is equivalent to requiring that the inverse measure of $\mu_\omega$ is discrete. We notice that while it is not hard to construct examples with $\dim_H X_\omega<1$ (which implies \eqref{lebesgue measure 0}), no example with $\dim_H X_\omega=1$ and $\mathrm{Leb}(X_\omega)=0$ is known.

\begin{theorem}\label{multifractal inverse}
	For $\mathbb P$-a.e.  $\omega\in\Omega$, let $\nu_{\omega}$ be the inverse measure of $\mu_{\omega}$. We have the following properties:
	\begin{enumerate}
		\item The multifractal formalism holds for $\nu_\omega$, with $\tau_{\nu_{\omega}}=\min(\mathcal{T},0)$.
		
		\item
		\begin{itemize}
			\item For any $d\in [\mathcal{T}'(+\infty),\mathcal{T}'(-\infty)]$, one has
			$\dim_H\overline{E}(\nu_{\omega},d)=\dim_HE(\nu_{\omega},d)=\mathcal{T}^*(d).$ In particular, the strong multifractal formalism holds on the maximal set  $\mathbb{R}\setminus \{d\in\mathbb{R} : \mathcal T^*(d)<(\min(\mathcal{T},0))^*(d)\}$.
			\item For any $d\in (0,\mathcal{T}'(+\infty))$, the sets $\overline{E}(\nu_{\omega},d)$ and $E(\nu_{\omega},d)$ are empty.
			\item For $d=0$,
			$	\overline{E}(\nu_{\omega},0)=E(\nu_{\omega},0)=\{\rm{atoms\  of\  }\nu_\omega\}$.		
		\end{itemize}
	\end{enumerate}
\end{theorem}

\begin{remark}\label{rem1}(1) The flavor of theorem~\ref{multifractal inverse}(i) is similar to that of  the result obtained in~\cite{BS2009} for  the inverse of Gibbs measures on cookie-cutter sets: for the level sets of  the lower local dimension, the Hausdorff spectrum is composed of two parts: a linear part with slope $\dim_H X_{\omega}$, which is established thanks to conditioned ubiquity theory, and a concave part which mainly reflects the  multifractal structure  of weak Gibbs measures or, equivalently, ratios of Birkhoff averages.  The properties stated in Theorem~\ref{multifractal inverse} are not considered in  \cite{BS2009}. Also, in \cite{BS2009} the level set  $\underline E(\nu_\omega,\mathcal{T}'(-\infty))$, corresponding to the maximal lower local dimension, is not treated when $\mathcal{T}^*(\mathcal{T}'(-\infty))=0$.
	
	(2) Although the main lines of the proof of  theorem~\ref{multifractal inverse}(i) are similar to those used to treat the case of deterministic Gibbs measures \cite{BS2009}, the study of $\nu_\omega$ requires the tools developed in \cite{Yuan2016} to study the multifractal nature of random weak Gibbs measures. Also,  it is made structurally more complex because the weak Gibbs measures are constructed on a random subshift; this is reflected in the expression of $\nu_\omega$ as a weighted sum of Dirac masses (see propositions~\ref{definition nu} and~\ref{measure v}).  Moreover, we need to establish a version of the heterogeneous ubiquity theorem of \cite{BS2007} adapted to our more general context. Indeed, until now the sufficient conditions required to directly apply the main result of \cite{BS2007} have been checked to hold in connection with a random Gibbs measure only in the random fullshift case and when $\mathbb{P}$ is a product measure \cite{BS2005}, and our investigations lead us to conclude that such conditions can be verified in our more general setting only if the dynamical system $(\Omega,\sigma,\mathbb{P})$ possesses rapid decay of correlations.
	
	(3) It is easy to see that there is a strong relationship between $T$ and $\mathcal{T}$: a short calculation yields $\mathcal{T}^*(d)=d T^*(1/d)$ for all $d$ in the support of $\mathcal T^*$. 
\end{remark}

The rest of the paper is organized as follows. Section 3  provides  an explicit writing of the measure $\nu_\omega$ and some useful estimate of the mass of its atoms. In section 4, we start the multifractal analysis  of $\nu_\omega$ by examining the possible scenarios which lead to a given lower local dimension. This yields a first, not everywhere sharp, but very useful for the sequel, upper bound for the lower Hausdorff spectrum. Indeed, it is already related to conditioned ubiquity properties associated with the sets of atoms, and thus it provides a beginning of concrete explanation of the origin of the linear part in the lower Hausdorff spectrum. In section 5, we derive the sharp upper bound for the $L^q$-spectrum of $\nu_\omega$, in which ubiquity properties remain hidden.  Section 6 introduces basic properties related to the approximation of $(\Phi,\Psi)$ by H\"older continuous random potentials, and used in next sections. Sections 7 to 9 obtain the sharp lower bound for the lower Hausdorff spectrum. This, combined with the result of section 5 gives the equality $\dim_H \underline E(\mu,d)=\tau_{\nu_\omega}^*(d)$ for all $d\in \mathbb{R}$, hence the validity of the multifractal formalism, as well as the equality $\tau_{\nu_\omega}=\min(\mathcal{T},0)$ by the duality property of Legendre transforms of concave functions. Specifically, section 7 derives the sharp lower bound for the lower Hausdorff spectrum in its non linear part, section 8  provides the conditioned ubiquity theorem used in section 9 to get the sharp lower bound for the lower Hausdorff spectrum in the linear part. Finally, section 10 deals with the Hausdorff dimension of the level sets $E(\nu_\omega,d)$ and $\overline E(\nu_\omega,d)$.

\section{Writing of $\nu_\omega$ as a sum of Dirac masses and estimates for the point masses}\label{section:The basic notations} We begin with a useful proposition established in~\cite{Yuan2016}, which provides estimates of the $\mu_\omega$ mass and the diameter of any set of the form $X_\omega^v$:

\begin{proposition}[\cite{Yuan2016}, Proposition 3]\label{for n}
	For $\mathbb{P}$-a.e.  $\omega\in \Omega$, there are  non increasing sequences $(\epsilon(\psi,\omega,n))_{n\ge 0}$ and $(\epsilon(\phi,\omega,n))_{n\ge 0}$, that we also denote as $(\epsilon(\Psi,\omega,n))_{n\ge 0}$ and $(\epsilon(\Phi,\omega,n))_{n\ge 0}$, converging to 0 as $n\to +\infty$, such that for all $n\in \mathbb{N}$, for all $v=v_0v_1\dots v_{n}\in\Sigma_{\omega,n}$, we have (the diameter of a set $E$ is denoted by $|E|$):
	\begin{enumerate}
		\item  For all $z\in \widering{U}_{\omega}^v,$
		$$
		\exp(S_n\psi(\omega,z)-n\epsilon(\psi,\omega,n))\leq |U_{\omega}^v|\leq \exp(S_n\psi(\omega,z)+n\epsilon(\psi,\omega,n)),$$
		hence for all $\vl\in [v]_{\omega}$,
		$$\exp(S_n\Psi(\omega,\vl)-n\epsilon(\Psi,\omega,n))\leq |U_{\omega}^v|\leq \exp(S_n\Psi(\omega,\vl)+n\epsilon(\Psi,\omega,n)).$$
		Consequently, for all $\vl\in X_{\omega}^v$:
		$$
		|X_{\omega}^v|\leq|U_{\omega}^v|\leq \exp(S_n\Psi(\omega,\vl)+n\epsilon(\Psi,\omega,n)).$$

		\item For any $\Upsilon\in\mathbb{L}^1_{\Sigma_{\Omega}}(\Omega,C(\Sigma))$, for any $v\in\Sigma_{\omega,n}$ with $n\in\mathbb{N}$,
		we have 
		$$
		\exp(-n\epsilon(\Upsilon,\omega,n))\leq \frac{\widetilde\mu^{\Upsilon}_{\omega}([v]_{\omega})}{\exp(S_n\Phi(\omega,\vl)-nP(\Upsilon))}
		\leq \exp(n\epsilon(\Upsilon,\omega,n)),
		$$
		for all $\vl\in[v]_{\omega}$, where $P(\Upsilon)$ denote the topological pressure for $\Upsilon$.
		
		Noticing $P(\Phi)=0$, we get 
		$$
		\exp(S_n\Phi(\omega,\vl)-n\epsilon(\Phi,\omega,n))\leq \widetilde\mu_{\omega}([v]_{\omega})\leq \exp(S_n\Phi(\omega,\vl)+n\epsilon(\Phi,\omega,n)),$$
		hence for all $z\in U_{\omega}^v$,
		$$
		\exp(S_n\phi(\omega,z)-n\epsilon(\phi,\omega,n))\leq \mu_{\omega}(X_{\omega}^v)\leq \mu_{\omega}(U^v_\omega),$$ as well as $\mu_{\omega}(U^v_\omega)\le \exp(S_n\phi(\omega,z)+n\epsilon(\phi,\omega,n))$ if $\widetilde\mu_\omega$ is atomless.
	\end{enumerate}
\end{proposition}

Proposition~\ref{for n} (ii) and assumption \eqref{cphi} imply that $\widetilde\mu_\omega$ is atomless $\mathbb P$-almost surely. Without loss of generality, we assume that this is the case for all $\omega\in\Omega$ and the sequences $(n\epsilon(\Psi,\omega,n))_{n\ge 0}$ and $(n\epsilon(\Phi,\omega,n))_{n\ge 0}$ are increasing as $n$ increasing to $\infty$. Furthermore, we can also ask $\sum_{i=0}^{n-1}\text{var}_{n-i}\Upsilon(\sigma^i\omega)\leq\epsilon(\Phi,\omega,n)$ for any $\Upsilon\in\mathbb{L}^1_{\Sigma_{\Omega}}(\Omega,C(\Sigma))$.

\medskip

Next we introduce a few notations required to explicitly write $\nu_\omega$ as a sum of weighted Dirac measures.

For $\omega\in\Omega$,  $n\ge 1$, $v\in \Sigma_{\omega,n}$ and $k\ge 1$ we define
$$S(\omega,v,k)=\{w\in \Sigma_{\sigma^n\omega,k}: vw\in \Sigma_{n+k}(\omega)\},$$
the set of words in $\Sigma_{\sigma^n\omega,k}$ which can be a suffix of $v$. Next we consider the set of words $w$ in $S(\omega,v,k)$ such that $U^{vw}$ has a right neighboring interval $U^{v\tilde w}$, with $\tilde w\in S(\omega,v,k)$ :
$$S'(\omega,v,k)=\left\{w\in S(\omega,v,k)\left|
\begin{array}{l}
\text{ there exists }\widetilde{w}\in S(\omega,v,k) \text{(necessarily unique)}\\\text{such that }
U_{\omega}^{v\widetilde{w}} \text{ is the nearest right  neighboring} \\\text{interval of $U_{\omega}^{vw}$}
\end{array}\right.
\right\}.$$
\begin{remark}Notice that $S'(\omega,v,k)$ may be empty, while this is not possible in the deterministic case, i.\,e. when the attractor is a cookie cutter set.
\end{remark}
For any $w\in S'(\omega,v,k)$, we denote by $\widetilde{w}$ the element of $S(\omega,v,k)$ such that $U_{\omega}^{v\widetilde{w}}$ is the closest right  neighboring interval of $U_{\omega}^{vw}$.

For every $v\in \Sigma_{\omega,\ast}$, $k\ge 1$ and $ w\in S(\omega,v,k)$, define
$$
m^{vw}_{\omega}=\min X_\omega^{vw}\quad\text{and}\quad  M^{vw}_{\omega}=\max X^{vw}_{\omega},
$$

For any $v\in \Sigma_{\omega,\ast}$, we define
$$I_{\omega}^v:=[F_{\mu_{\omega}}(m_{\omega}^v),F_{\mu_{\omega}}(M_{\omega}^v))
=F_{\mu_{\omega}}(X_{\omega}^v)\setminus\{F_{\mu_{\omega}}(M_{\omega}^v)\}.$$

Since the support of $\mu_{\omega}$ restricted to the interval $[m_{\omega}^v,M_{\omega}^v]$ (or $U_{\omega}^v$) is $X_{\omega}^v$, by construction $I_{\omega}^v$ is a non-empty interval of length $|I_{\omega}^v|=\mu_{\omega}(X_{\omega}^v)=\widetilde\mu_{\omega}([v]_{\omega})$.

Also, since $\text{supp}(\mu_{\omega})= X_{\omega}$ and $\bigcup_{v\in\Sigma_{\omega,n}}X_{\omega}^v=X_{\omega}$, the families of intervals
$\mathcal{F}_{\omega}^{n}=\{I_{\omega}^v\}_{v\in\Sigma_{\omega,n}},\ n\geq 1$, form a sequence of refined partitions   of $[0,1)$ into intervals.

For any $v\in \Sigma_{\omega,\ast}$ and $s\in S'(\omega,v,1)$, we define
$$
x_{\omega}^{vs}=F_{\mu_{\omega}}(M_{\omega}^{vs}).
$$
We also define $m_{\omega}^{\min}=\min X_{\omega}$ and $M_{\omega}^{\max}= X_{\omega}$.

Notice that by construction we have
$$
\nu_\omega(\widering{I}_{\omega}^{v})\leq |X_{\omega}^v|\leq |U_{\omega}^v|$$
for any $\vl\in [v]_{\omega}$ and $v\in \Sigma_{\omega,*}$.

We can now give the following explicit form for the inverse of the random weak Gibbs measures $\{\mu_{\omega}:\omega\in \Omega\}$.
\begin{proposition}[{\rm The inverse measure $\nu_{\omega}$ of  $\mu_{\omega}$}]\label{definition nu}
	The inverse measure $\nu_{\omega}$  of the random weak Gibbs measure $\mu_{\omega}$ is the discrete probability  measure on $[0,1]$ given by the following weighted sum of Dirac measures:
	\begin{equation}\label{eq1.1}
	\nu_{\omega}=m_{\omega}^{\min}\cdot\delta_0+\sum_{v\in\Sigma_{\omega,*}}\sum_{s\in S'(\omega,v,1)}(m_{\omega}^{v\widetilde{s}}-M_{\omega}^{vs})\cdot \delta_{x_{\omega}^{vs}}+(1-M_{\omega}^{\max})\delta_1.
	\end{equation}
\end{proposition}

This proposition can be  easily proved if we notice the following two facts. On the one hand, from the definition we can get that at each point $x_{\omega}^{vs}$, the point mass  is at least $m_{\omega}^{v\widetilde{s}}-M_{\omega}^{vs}$.
On the other hand, we know the total mass of $\nu_\omega$ is 1 and $m_{\omega}^{\min}+\sum_{v\in\Sigma_{\omega,*}}\sum_{s\in S'(\omega,v,1)}(m_{\omega}^{v\widetilde{s}}-M_{\omega}^{vs}) +(1-M_{\omega}^{\max})=1$ from equation~\eqref{lebesgue measure 0}.

\begin{remark}
	Notice that even if $S'(\omega,v,1)\neq\emptyset$, the weight $m_{\omega}^{v\widetilde{s}}-M_{\omega}^{vs}$  may vanish if there is no gap between $X_\omega^{vs}$ and $X_\omega^{v\widetilde{s}}$. For instance, it is not difficult to see that in the full shift case this situation occurs  with probability 1,  infinitely many times,  if and only if with probability 1 we have $l(\omega)\ge 3$, $\min (U_\omega^{1})=0$, $\max (U_\omega^{l(\omega)})=1$, and $\{1\le i\le l(\omega)-1: U_\omega^{i} \text{ and }U_\omega^{i+1}\text{ are contiguous}\}\neq\emptyset$.
\end{remark}

We end this section with a non trivial lower bound estimate for some point masses associated with $\nu_\omega$ (Proposition~\ref{measure v}). For every $k\ge 1$  define
$$
\text{gap}(\omega,k)=\inf_{v\in \Sigma_{\omega,1}}\sup_{1\le m\leq k}\sup_{w \in S'(\omega,v,m)}\{m^{v\widetilde{w}}_{\omega}-M^{vw}_{\omega}\}.
$$

\begin{lemma}\label{gap} We have
	$$\mathbb{P}(\{\omega\in\Omega:\sup_{k\geq 1}\mathrm{gap}(\omega,k)>0\})>0.$$
	Consequently,  setting $\mathrm{Gap}(k,\gamma)=\{\omega\in\Omega:\mathrm{gap}(\omega,k)>\gamma\}$, there exist some $k_{\psi}>0$ and $\gamma_\psi>0$ such that $\mathbb{P}(\mathrm{Gap}(k_{\psi},\gamma_\psi))>0$.
\end{lemma}
\begin{remark}
	We notice that property~\eqref{lebesgue measure 0} is not necessary to get lemma~\ref{gap}. We only need that $X_{\omega}$ differs from $[0,1]$ for $\mathbb P$-a.e. $\omega\in\Omega$.
\end{remark}
\begin{proof} Suppose, by contradiction, that  the result does not hold. Then for $\mathbb{P}$-a.e.  $\omega$,
	there exists $v\in \Sigma_{\omega,1}$ such that
	$$\sup_{m\in\mathbb{N}}\sup_{w \in S'(\omega,v,m)}m^{v\widetilde{w}}_{\omega}-M^{vw}_{\omega}=0.$$
	This implies that $X_{\omega}^v$ has no gap. Then $X_{\omega}^v$ is either a point or an interval. Since $X_{\omega}$ has a Lebesgue measure 0, we get that it is a point.
	
	Now, defining
	$$B=\{\omega\in\Omega:\ M(\omega)\leq M',\ l(\omega)\geq 2\},$$
	we have $\mathbb{P}(B)>0$ for $M'$ large enough. For any $\omega\in\Omega$, define $b_k(\omega)$ the $k$-th return time of $\omega$ to the set $B$ by the map $\sigma$. From ergodic theorem we have $\lim_{k\to\infty}\frac{b_k(\omega)}{k}=\frac{1}{\mathbb{P}(B)}$ for $\mathbb{P}$-a.e. $\omega\in \Omega$.
	Define $\Omega'=\{\omega\in \Omega:\lim_{k\to\infty}\frac{b_k(\omega)}{k}=\frac{1}{\mathbb{P}(B)} \}$.
	
	For any $\omega\in\Omega'$, we know that there are at least four words in $\Sigma_{\omega,b_{M'+2}}$ with the prefix $v\in\Sigma_{\omega,1}$, and  we denote them by $w^1,w^2,w^3$ and $w^4$. We can assume that these intervals appear from the left to the right as $U_{\omega}^{w^1},U_{\omega}^{w^2},U_{\omega}^{w^3}$ and $U_{\omega}^{w^4}$. The sets $X^{w^i}_{\omega}\subset U_{\omega}^{w^i}, i=1,2,3,4$, are not empty since by definition the random transition matrix $A$ has
	at least one non-zero entry in each row and each column. Choose $x_i\in X^{w^i}_{\omega} \subset U_{\omega}^{w^i}, i=1,2,3,4$. Since $U_{\omega}^{w^i}, i=1,2,3,4$ are intervals, we have that $x_4-x_1>0$, which contradicts the fact that $X_{\omega}^v$ is a singleton.
\end{proof}

\begin{proposition}\label{measure v}
	For $\mathbb{P}$-a.e.  $\omega\in \Omega$, for all $n\in \mathbb{N}$, for all $v\in \Sigma_{\omega,n}$, there exists $k_v$ and $w\in S'(\omega,v,k_v)$ such that
	$$m_{\omega}^{v\widetilde{w}}-M_{\omega}^{vw}\geq \exp(S_n\Psi(\omega,\vl)-o(n))$$
	for any $\vl\in [v]_{\omega}$.
	Here $o(n)$ is independent of $v$, and we have $k_v=o(n)$  independently on $v$ as well.
\end{proposition}

\begin{proof}
	For any $N\in \mathbb{N}$, let
	$$\varOmega_N=\left \{\omega:\ M(\omega)< N,\  \frac{1}{n}\sum_{k=0}^{n-1}\sup_{1\leq s\leq l(\sigma^k \omega)}\sup_{x\in U_{\sigma^k\omega}}|\psi(\omega,s,x)|\leq 2C_{\psi},\forall n\geq N\right \}.$$
	Choose $N$ large enough so that $\mathbb{P}(\varOmega_N)>0$.
	
	For $\mathbb{P}$-a.e.  $\omega\in \Omega$, for $n$ large enough, denote by  $H_1(n)$ the smallest integer such that $\sigma^{n+H_1(n)}\omega\in \varOmega_N$, and $H_2(n)$ the smallest integer such that $\sigma^{n+H_1(n)+H_2(n)}\omega\in \mathrm{Gap}(k_{\psi},\gamma_\psi)$ with $H_2(n)\geq N$. Since $\mathbb{P}(\varOmega_N)>0$ and $\mathbb{P}(\mathrm{Gap}(k_{\psi},\gamma_\psi))>0$, from ergodic theorem we can get that $\lim_{n\to \infty}\frac{H_1(n)+H_2(n)}{n}=0$. Moreover, since $\sigma^{n+H_1(n)+H_2(n)}\omega\in \mathrm{Gap}(k_{\psi},\gamma_\psi)$, there exists $1\leq s\leq l(\sigma^{n+H_1(n)+H_2(n)}\omega)$ and $v'\in S'(\sigma^{n+H_1(n)+H_2(n)}\omega,s,k)$ with $k\leq k_{\psi}$ such that $m_{\sigma^{n+H_1(n)+H_2(n)}\omega}^{s\widetilde{v'}}-M_{\sigma^{n+H_1(n)+H_2(n)}\omega}^{sv'}>\gamma_\psi$. 
	
	For any $v\in \Sigma_{\omega,n}$, since $M(\sigma^n\omega)\leq M(\sigma^{n+H_1(n)}\omega)+H_1(n)<H_1(n)+H_2(n)$, recalling the operation $*$ defined in section~\ref{subsection: Random subshift, relativized entropy, topological pressure  and weak Gibbs measures}, there exists a word $s$ such that $v\ast s\in \Sigma_{\omega,n+H_1(n)+H_2(n)+1}$, and by construction $U_{\omega}^{v\ast sv'}$ and $U_{\omega}^{v\ast s\widetilde{v'}}$ are contiguous intervals. We simply denote the left one by $U_{\omega}^{vw}$ and the right by $U_{\omega}^{v\widetilde{w}}$.  We have $T_{\omega}^{v\ast}([M_{\omega}^{vw},m_{\omega}^{v\widetilde{w}}])=[M_{\sigma^{n+H_1(n)+H_2(n)}\omega}^{sv'}, m_{\sigma^{n+H_1(n)+H_2(n)}\omega}^{s\widetilde{v'}}]$. Now using Lagrange's finite-increment theorem we can get 
	$$m_{\omega}^{v\widetilde{w}}-M_{\omega}^{vw}\geq \gamma_\psi \exp(S_n\Psi(\omega,\vl)-o(n)).$$
	for any $\vl\in [v]_{\omega}$, since $H_1(n)+H_2(n)$ is a $o(n)$. Then the result holds since $\gamma_\psi$ is a constant. Moreover $k_v=|w|=H_1(n)+H_2(n)+k\leq H_1(n)+H_2(n)+k_{\psi}$ is also $o(n)$.
\end{proof}

\begin{remark}\label{rem3}Proposition \ref{measure v} implies that for $\mathbb{P}$-a.e.  $\omega\in \Omega$, for all $n\in \mathbb{N}$, for any $v\in \Sigma_{\omega,n}$, there exist some point $x$ of the form $x_{\omega}^{vw}\in\widering{I}_{\omega}^{v}$, with $|w|=o(n)$, such that
	$$
	\vw(\{x\})\geq \exp(S_n\Psi(\omega,\vl)-o(n))
	$$
	for any $\vl\in [v]_{\omega}$. For each $v\in \Sigma_{\omega,\ast}$, we fix such a point and denote it by $z_{\omega}^v$. These points will play a crucial role in proving the sharp lower bound for the lower Hausdorff spectrum of $\nu_\omega$.
\end{remark}
Arguments similar to  those leading to proposition~\ref{measure v} lead to the following remark.
\begin{remark}
	For $\mathbb{P}$-a.e. $\omega\in\Omega$, for all $n\in \mathbb{N}$ and $v\in \Sigma_{\omega,n}$, for any $\vl\in [v]_{\omega}$,
	\begin{equation*}\label{}
	|X_{\omega}^v|\geq \exp(S_n\Psi(\omega,\vl)-o(n)),
	\end{equation*}
	where the $o(n)$ does not depend on the choice of $v$. Consequently, we have
	\begin{equation*}\label{}
	\exp(S_n\Psi(\omega,\vl)-o(n))\leq |X_{\omega}^v|\leq \exp(S_n\Psi(\omega,\vl)+o(n)).
	\end{equation*}
\end{remark}

Next section provides first information on the lower local dimension of $\nu_\omega$ and the lower Hausdorff spectrum.

\section{Pointwise behavior of $\nu_\omega$ and an upper bound for the lower Hausdorff spectrum without using the multifractal formalism}\label{section:Pointwise behavior} The following definitions will be essential in making explicit the connection between the lower local dimension of $\nu_\omega$ and the conditioned ubiquity which partly governs the multifractal structure of $\nu_\omega$.
\begin{definition}\label{def 4.1}
	For $v\in \Sigma_{\omega,\ast}$, we set
	$$\ell_{\omega}^v=2|I_{\omega}^v|=2\widetilde\mu_{\omega}([v]_{\omega}),\quad
	\alpha_\omega^v=\frac{\widetilde{\Psi}(\omega,v)}{\log|I_{\omega}^v|},$$
	where $$\widetilde{\Psi}(\omega,v)=\sup_{\underline{v}\in [v]_{\omega}}\{S_{|v|}\Psi(\omega,\underline{v})\}$$
	(we define $\widetilde\Phi$ similarly).
	We notice that due to proposition~\ref{for n}(2) we have
	\begin{equation}\label{alphaomega}
	\alpha_\omega^v=\frac{\widetilde{\Psi}(\omega,v)}{\widetilde{\Phi}(\omega,v)}+o(1),
	\end{equation}
	where $o(1)$ tends uniformly in $v$ to 0 as $|v|$ tends to $\infty$.		
	
	For $x\in[0,1)$ and $n\geq 1$, let $v(\omega,n,x)$ be the unique element $v$ in $\Sigma_{\omega,n}$ such that $x\in I_{\omega}^{v}$. If $x=1$,  $v(\omega,n,1)$ is the unique $v\in\Sigma_{\omega,n}$ such that $1\in \overline{I_{\omega}^{v}}$.  If there is no confusion we will denote $v(\omega,n,x)$ by $v(n,x)$ or $x|_n$ for short.
	
	Define 		$$\alpha_{\omega}^n(x)=\alpha_{\omega}^{x|_n}\quad\text{and}\quad
	\alpha_{\omega}(x)=\liminf_{n\rightarrow\infty}\alpha_{\omega}^n(x).$$
	
	
	For $x\in[0,1]\setminus\{x_{\omega}^{vs}:v\in\Sigma_{\omega,\ast},s\in S'(\omega,v,1)\}$, the approximation degree 
	$\xi_{\omega}^x$ by the system $\{(x_{\omega}^{vs},\ell_{\omega}^v)\}_{v\in\Sigma_{\omega,\ast},s\in S'(\omega,v,1)}$ is defined as
	$$
	\xi_{\omega}^x=\limsup_{n\rightarrow\infty}\sup_{s\in S'(\omega,x|_n,1)}\frac{\log|x-x_{\omega}^{x|_n s}|}{\log\ell_{\omega}^{x|_n}}.
	$$	\end{definition}
Set 
$$
\Xi_\omega=\{x_{\omega}^{vs}:v\in\Sigma_{\omega,\ast},s\in S'(\omega,v,1)\}$$ 
and  
$$
\Xi'_\omega=\{\text{atoms of }\nu_\omega\}=\{x_{\omega}^{vs}:v\in\Sigma_{\omega,\ast},s\in S'(\omega,v,1) ,\text{ and } m_{\omega}^{v\widetilde{s}}-M_{\omega}^{vs}>0\}\subset \Xi_\omega.
$$

\begin{proposition}\label{Compare}
	\begin{enumerate}
		\item If $x\in\Xi'_\omega$, then $\nu_{\omega}(\{x\})>0$, thus $\dm({\nu_{\omega}},x)=0$.
		
		\item For any $x\in [0,1]$, if
		$x\notin \Xi_\omega$, then
		$$
		\frac{\aw(x)}{\xi_{\omega}^x}\leq \dl(\vw,x)\leq\aw(x),$$
		with the convention that  
		if $\xi_{\omega}^x=+\infty$ then $\frac{\alpha_{\omega}(x)}{\xi_{\omega}^x}=0$.
	\end{enumerate}

\end{proposition}
\begin{proof}
	$(i)$ is obvious. Let us prove $(ii)$. 
	
	Let $x\in [0,1]\setminus \Xi_\omega$ and $r>0$. If $r$ is small enough, the integer
	\begin{equation}\label{definition n}
	n_{\omega}^{x,r}=\max\{n:\exists v\in\Sigma_{\omega,n}\text{ such that }B(x,r)\subset I_{\omega}^{v}\}
	\end{equation}
	is well defined, and by definition of $c_\psi$ (see \eqref{cpsi}) we have  $n_{\omega}^{x,r}\leq \frac{-2\log r}{c_{\psi}}$. Moreover, $n_{\omega}^{x,r}\to \infty $ as $r\to 0$.
	By definition of $n_{\omega}^{x,r}$, if we denote by $v(x,r)$ the word  $v(\omega,n_{\omega}^{x,r},x)$,  there exists an unique  element $s$ of $S'(\omega,v(x,r),1)$ such that $$x_{\omega}^{v(x,r)s}\in B(x,r)\subset I_{\omega}^{v(x,r)}.$$
	The inclusion $B(x,r)\subset I_{\omega}^{v(x,r)}$ and proposition~\ref{for n}(i) imply
	\begin{eqnarray*}
		&&\nu_{\omega}(B(x,r/2))\leq\nu_{\omega}(I_{\omega}^{v(x,r)}\setminus\{\max I_{\omega}^{v(x,r)},\min I_{\omega}^{v(x,r)}\})\leq|X^{v(x,r)}_{\omega}| \leq |U_{\omega}^{v(x,r)}|\\
		&\leq& \exp(S_{|v(x,r)|}\Psi(\omega,\underline{v})+|v(x,r)|\epsilon(\Phi,\omega,|v(x,r)|)) \text{ where } \vl\in [v(x,r)]_{\omega}.
	\end{eqnarray*}

	%
	Now for any $\varepsilon>0$, by definition of $\xi_{\omega}^x$, for $r$ small enough we have $$r\geq |x-x_{\omega}^{v(x,r)s}|\geq(2|I_{\omega}^{v(x,r)}|)^{\xi_{\omega}^x+\varepsilon}.$$
	
	Moreover, again for $r$ small enough, we have $$\exp(\widetilde{\Psi}(\omega,v(x,r)))\leq |I_{\omega}^{v(x,r)}|^{\aw(x)-\varepsilon}$$
	by definition of $\aw(x)$. These estimates yield
	$$\vw(B(x,r/2))\leq \exp(\widetilde{\Psi}(\omega,v(x,r))+o(|v(x,r)|))\leq  r^{\frac{\aw(x)-\varepsilon}{\xi_{\omega}^x+\varepsilon}}\exp(o(n_{\omega}^{x,r})),$$
	and by letting $r$ tend to zero, since $n_{\omega}^{x,r}\leq \frac{-2\log r}{c_{\psi}}$, it follows that $ \dl(\vw,x)\geq\frac{\aw(x)-\varepsilon}{\xi_{\omega}^x+\varepsilon}$. From the arbitrariness of $\varepsilon$ we get that $\dl(\vw,x)\geq\frac{\aw(x)}{\xi_{\omega}^x}$.
	
	For the second inequality, let $\{p_i\}_{i\geq 1}$ be an increasing sequence of integers
	such that $\exp(\widetilde{\Psi}(\omega,x|_{p_i}))\geq |I_{\omega}^{x|_{p_i}}|^{\aw(x)+\epsilon}$ for all $i\geq 1$. Now recall remark~\ref{rem3}. Since $z_{\omega}^{x|_{p_i}}\in B(x,2|I_{\omega}^{x|_{p_i}}|)$, we have
	\begin{eqnarray*}
		\vw(B(x,2|I_{\omega}^{x|_{p_i}}|)) &\geq& \vw(\{z_{\omega}^{x|_{p_i}}\})\geq\exp(S_{p_i}\Psi(\omega,\vl))\exp(-o(p_i)) \\
		&\geq& |I_{\omega}^{x|_{p_i}}|^{\aw(x)+\epsilon}\exp(-o(p_i)).
	\end{eqnarray*}
	Also, $|I_{\omega}^{x|_{p_i}}|\leq \exp(-\frac{c_{\psi} p_i}{2})$ for $p_i$ large enough and $\epsilon$ is arbitrarily small. Consequently,   $\dl(\vw,x)\leq\aw(x).$
\end{proof}

\begin{remark}\label{limsupalpha}
	Arguments similar to those used to get  proposition~\ref{Compare} show that $$\overline{\dim}_{\mathrm{loc}}(\nu_\omega,x)\le \limsup_{n\to\infty}\alpha^n_\omega(x).$$
\end{remark}	

\begin{definition}
	Let $\alpha >0, \xi\geq 1$ and $\varepsilon>0$. A real number $x\in [0,1]$ is said to satisfy the property $\mathcal{P}(\omega,\alpha,\xi,\varepsilon)$ if there exists an increasing sequence of positive integers $(n_k)_{k\geq 1}$ such that for every $k\geq 1$, there exists $v\in \Sigma_{\omega,n_k}$ and $s\in S'(\omega,v,1)$, such that
	$x\in B(x_{\omega}^{vs},(\ell_{\omega}^{vs})^{\xi-\varepsilon})$ and $\alpha_{\omega}^v\in [\alpha-\varepsilon,\alpha+\varepsilon]$.
\end{definition}

We now introduce new sets.
\begin{definition}\label{def710}
	For $d\ge 0$, let $$F(\omega,d)=\left\{x\in (0,1)\left|\begin{array}{l}
	\forall\varepsilon>0, \exists \alpha\in \mathbb{Q}^+,\exists\xi\in \mathbb{Q},\xi\geq 1 \text{ such that }\\
	\alpha/\xi\leq d+2\varepsilon \text{ and } x \text{ satisfies the property } \mathcal{P}(\omega,\alpha,\xi,\varepsilon)
	\end{array}\right.\right\}.$$
\end{definition}

Now, the following proposition explores the relationship between the level sets $\El(\vw,d)$ and the sets $F(\omega,d)$.  
\begin{proposition}\label{prop8.3}
	For $\mathbb{P}$-a.e.  $\omega$, for any $ d\geq 0$, we have $(\El(\vw,d)\setminus\Xi_\omega\subset F(\omega,d)$.
\end{proposition}

\begin{proof}
	Fix $d\ge 0$, $x\in \El(\vw,d)$ and $\varepsilon>0$. By definition of $\underline{\dim}_{\mathrm{loc}}(\nu_\omega,x)$, there exists a sequence $(r_k)_{k\geq 1}$ of positive numbers decreasing to zero such that for all $k\geq 1$ we have $\vw(B(x,r_k/2))\geq (r_k/2)^{d+\varepsilon}$. Now recall the definition of $n_{\omega}^{x,r}$ given in~\eqref{definition n}. If $x\notin \{x_{\omega}^{vs}:v\in\Sigma_{\omega,\ast},s\in S'(\omega,v,1)\}$, then $n_{\omega}^{x,r}\to \infty$ as $r\to 0$.
	Since $n_{\omega}^{x,r}$ is maximal, there exist $v=v(x,r)$ and  $s\in S'(\omega,v(x,r),1)$ such that $$x_{\omega}^{v(x,r)s}\in B(x,r)\subset I_{\omega}^{v(x,r)}.$$
	Then $\nu_{\omega}(B(x,r/2))\leq \nu_{\omega}(I_{\omega}^{v(x,r)}\setminus\{\max I_{\omega}^{v(x,r)},\min I_{\omega}^{v(x,r)}\})\leq |U_{\omega}^{v(x,r)}|$, so
	$$(r_k/2)^{d+\varepsilon}\leq \nu_{\omega}(B(x,r_k/2))\leq\exp(\widetilde{\Psi}(\omega, v(x,r_k))+o(|v(x,r_k)|)), $$
	where in the last inequality we used proposition~\ref{for n}.
	Consequently, due to the definition of $\alpha_\omega^v$ and proposition~\ref{for n} again, we have
	$$|I_{\omega}^{v(x,r_k)}|^{\alpha_{\omega}^{v(x,r_k)}+o(1)}=\exp(\widetilde{\Psi}(\omega, v(x,r_k)))\geq (r_k/2)^{d+\varepsilon} .$$
	Since $x_{\omega}^{v(x,r)s}\in B(x,r)$, $|x-x_{\omega}^{v(x,r_k)s}|\leq r_k$.  Writing $$r_k\geq|x-x_{\omega}^{v(x,r_k)s}|=(2(|I_{\omega}^{v(x,r_k)}|))^{\xi_k}\geq 2(|I_{\omega}^{v(x,r_k)}|)^{\xi_k},$$ we get
	$$|I_{\omega}^{v(x,r_k)}|^{\alpha_{\omega}^{v(x,r_k)}+o(1)}\geq (|I_{\omega}^{v(x,r_k)}|)^{\xi_k(d+\varepsilon)},\text{ and }\xi_k\geq 1.
	$$
	
	If $\limsup_{k\to \infty}\xi_k< \infty$, there exists $(\alpha,\xi)\in \mathbb{Q}^+\times(\mathbb{Q}\cap[1,+\infty))$ and an increasing sequence of integers $(k_s)_{s\geq 1}$ such that
	$$|\alpha_{\omega}^{v(x,r_{k_s})}-\alpha|\leq \varepsilon,  |\xi_{k_s}-\xi|\leq \varepsilon,\text{ and }
	\alpha/\xi\leq d+2\varepsilon.$$
	This means $x\in G(\omega,\alpha,\varepsilon,\xi)$,  $(\alpha,\xi)\in \mathbb{Q}^+\times(\mathbb{Q}\cap[1,+\infty))$ and $\alpha/\xi\leq h+2\varepsilon$.
	
	If $\limsup_{k\to \infty}\xi_k= \infty$, there exists $\alpha\in \mathbb{Q}^+$ and an increasing sequence of integer number $(k_s)_{s\geq 1}$ such that
	$$|\alpha_{\omega}^{v(x,r_{k_s})}-\alpha|\leq \varepsilon \text{ and }\xi_{k_s}\to \infty.$$
	Since $\alpha_{\omega}^{v(x,r_{k_s})}$ is bounded (for $\mathbb{P}$-a.e. $\omega$), there exists some $\xi\in \mathbb{Q}\cap[1,+\infty)$ with $\alpha/\xi\leq d+2\varepsilon$ such that $x$ satisfies $\mathcal{P}(\alpha,\varepsilon,\xi)$ (because if $\xi_1\leq \xi_2$ then $\mathcal{P}(\alpha,\varepsilon,\xi_2)$ implies $\mathcal{P}(\alpha,\varepsilon,\xi_1)$).
	
	Finally,
	\begin{equation}\label{inclusion}
	(\El(\vw,d)\setminus\Xi_\omega\subset F(\omega,d).
	\end{equation}
\end{proof}

\begin{definition}For every $\alpha,\varepsilon>0$ and $\xi\geq 1$, let
	$$G(\omega,\alpha,\varepsilon,\xi)=\bigcap_{N\geq1} \bigcup_{n\geq N} \bigcup_{v\in\Sigma_{\omega,n}, s\in S'(\omega,v,1):\alpha_{\omega}^v\in [\alpha-\varepsilon,\alpha+\varepsilon]}B(x_{\omega}^{vs},(\ell_{\omega}^{v})^{\xi}).$$
\end{definition}

It is easily seen that
$$
F(\omega,d)\subset \bigcup_{\alpha\in \mathbb{Q}^+}\bigcup_{\xi\in \mathbb{Q}\cap[1,+\infty),\alpha/\xi\leq d+2\varepsilon}G(\omega,\alpha,\varepsilon,\xi).$$

%
\begin{lemma}\label{control G}
	There exists $C>0$ such that  for $\mathbb{P}$-a.e. $\omega$, for $\varepsilon>0$ small enough, for all rationals $\alpha>0$ and $\xi\geq 1$,
	$$\dim_H G(\omega,\alpha,\varepsilon,\xi)\leq C\varepsilon+\frac{\max (\mathcal{T}^\ast(\alpha-\varepsilon),\mathcal{T}^\ast(\alpha),\mathcal{T}^\ast(\alpha+\varepsilon))}{\xi}. $$
\end{lemma}
\begin{proof} It is enough to prove the result for fixed $\epsilon>0$ and rational numbers $\alpha>0$ and $\xi\ge 1$.
	For any $N\geq 1$, let $\delta_N=\sup_{v\in \Sigma_{\omega,N}}\ell_{\omega}^{v}$. By construction, if $G(\omega,\alpha,\varepsilon,\xi)\neq\emptyset$, given $s\in\mathbb{R}$ we have
	$$\mathcal{H}_{\delta_N}^s(G(\omega,\alpha,\varepsilon,\xi))\leq \sum_{n\geq N}\sum_{v\in\Sigma_{\omega,n}: \alpha-\varepsilon\leq\alpha_{\omega}^v\leq\alpha+\varepsilon }l(\sigma^n\omega) 2^s (\ell_{\omega}^{v})^{s\xi}$$
	where we naturally extend the definition of  $\mathcal H^s_\delta$  to negative $s$.
	
	Here, to avoid confusions,  we recall that $l(\omega)$ is the number of types in the subshift, and $\ell_{\omega}^{v}$ is the length of the interval $I_{\omega}^v$.
	
	\medskip
	
	\noindent {\bf Case 1: $\alpha\leq \mathcal{T}'(0-)-\varepsilon$.} Since $\alpha_{\omega}^v=\frac{\widetilde{\Psi}(\omega,v)}{\log |I_{\omega}^v|}$, then for any $q\geq 0$ one has:
	\begin{eqnarray*}
		\mathcal{H}_{\delta_N}^s (G(\omega,\alpha,\varepsilon,\xi))&\leq & 2^s\sum_{n\geq N}\sum_{v\in\Sigma_{\omega,n}:q\widetilde{\Psi}(\omega,v)\geq q(\alpha+\varepsilon)\log |I_{\omega}^v|} l(\sigma^n\omega)(\ell_{\omega}^{v})^{s\xi}\\
		&\leq &   4^s\sum_{n\geq N}\sum_{v\in\Sigma_{\omega,n}} l(\sigma^n\omega)\exp(q\widetilde{\Psi}(\omega,v)-q(\alpha+\varepsilon)\log |I_{\omega}^v|)\\
		&&\cdot\exp(s\xi\log |I_{\omega}^v|).
	\end{eqnarray*}
	Now take $s=(\eta+\mathcal T^*(\alpha+\varepsilon))/\xi$ with $\eta>0$, we can get 
	$$\mathcal{H}_{\delta_N}^s (G(\omega,\alpha,\varepsilon,\xi))\leq \sum_{n\geq N}\sum_{v\in\Sigma_{\omega,n}}\exp(q\Psi(\omega,\vl)-(q(\alpha+\varepsilon)-\eta-\mathcal T^*(\alpha+\varepsilon))\Phi(\omega,\vl)) \cdot \exp(o(n))$$
	
	Since $\alpha\leq \mathcal{T}'(0-)-\varepsilon$, that is $\alpha+\varepsilon\leq \mathcal{T}'(0-)$, there exists $q\geq 0$ such that 
	$$\mathcal T^*(\alpha+\varepsilon)=(\alpha+\varepsilon)q-\mathcal T(q)-\gamma_{q},$$
	with $0\leq \gamma_q\leq \frac{\eta}{2}$.
	
	Then, 		
	\begin{eqnarray*}
		\mathcal{H}_{\delta_N}^s(G(\omega,\alpha,\varepsilon,\xi))
		\leq \sum_{n\geq N}\sum_{v\in\Sigma_{\omega,n}}\exp(q\Psi(\omega,\vl)-(\mathcal{T}(q)-\eta/2)\Phi(\omega,\vl)) \cdot \exp(o(n)),
	\end{eqnarray*}
	where $\underline v$ is any element of $[v]_\omega$.  Then
	\begin{eqnarray*}
		\mathcal{H}_{\delta_N}^s(G(\omega,\alpha,\varepsilon,\xi))
		&\leq& \sum_{n\geq N}\sum_{v\in\Sigma_{\omega,n}}\widetilde{\mu}_{\omega}^{q\Psi-\mathcal{T}(q)\Phi}([v]_{\omega})\exp(-\eta c_{\Phi} n/2+o(n)) \leq \sum_{n\geq N}\exp(-\frac{\eta c_{\Phi}}{4}n)
	\end{eqnarray*}
	for $n$ large enough (recall $c_{\Phi}=c_{\phi}>0$). Consequently,
	$\lim_{N\to\infty}\mathcal{H}_{\delta_N}^s G(\omega,\alpha,\varepsilon,\xi)=0$. However, if $\mathcal T^*(\alpha+\varepsilon)<0$, we can choose $\eta$ and $q$ such that $s<0$, in which case it is necessary that $\lim_{N\to\infty}\mathcal{H}_{\delta_N}^s =+\infty$ if $G(\omega,\alpha,\varepsilon,\xi)$ is not empty. Consequently, if $\mathcal T^*(\alpha+\varepsilon)<0$, then $G(\omega,\alpha,\varepsilon,\xi)=\emptyset$. Otherwise, $\dim_H G(\omega,\alpha,\varepsilon,\xi)\leq (\eta+\mathcal T^*(\alpha+\varepsilon))/\xi$. This holds for all $\eta>0$, so $\dim_H G(\omega,\alpha,\varepsilon,\xi)\le \mathcal T^*(\alpha+\varepsilon)$.
	
	\medskip
	
	\noindent {\bf Case 2: $\alpha\geq \mathcal{T}'(0+)+\varepsilon$.}
	It is almost the same as before except that one needs to use $q\leq0$ and
	$q\widetilde{\Psi}(\omega,v)\geq q(\alpha-\varepsilon)\log |I_{\omega}^v|$.
	
	\medskip
	
	\noindent {\bf Case 3:  $\alpha\in (\mathcal{T}'(0-)-\varepsilon, \mathcal{T}'(0+)+\varepsilon)$.}
	Two situations must be considered.If $\mathcal{T}'(0-)-\mathcal{T}'(0+)>0$, we can assume  $\varepsilon<\frac{\mathcal{T}'(0-)-\mathcal{T}'(0+)}{2}$. Then $\mathcal{T}'(0-)-\varepsilon>\mathcal{T}'(0+)+\varepsilon,$ so that Case 3 is empty. If $\mathcal{T}'(0-)=\mathcal{T}'(0+)$, then $\mathcal{T}$ is differentiable at $0$. Take $s=\frac{\eta+\mathcal{T}(0)}{\xi}$ with $\eta>0$. Then
	$$\mathcal{H}_{\delta_N}^sG(\omega,\alpha,\varepsilon,\xi)\leq \sum_{n\geq N}\exp(-\frac{\eta nc_{\Phi}}{2}-nP(\mathcal{T}(0)\Phi))=\sum_{n\geq N}\exp(-\frac{\eta nc_{\phi}}{2})<\infty.$$
	Here we used the fact that by definition we have $P(-\mathcal{T}(0)\Phi)=0$.
	
	This yields $\dim_H G(\omega,\alpha,\varepsilon,\xi)\leq \frac{-\mathcal{T}(0)}{\xi}$ since we can choose $\eta$ arbitrarily close to $0$. Since $\mathcal{T}^*$ is concave, for $\varepsilon$ small enough there exists some $C>0$ such that
	$\frac{-\mathcal{T}(0)}{\xi}=\frac{\mathcal{T}^*(\mathcal{T}'(0))}{\xi}\leq C\varepsilon+\frac{\mathcal{T}^*(\alpha)}{\xi}.$	
\end{proof}
As a consequence of proposition~\ref{prop8.3} and lemma~\ref{control G}, the following corollary will provide us with a first upper bound for $\dim_H \underline E(\nu_{\omega},d)$ which will turn out to be sharp on $ [0,\mathcal{T}'(t_0-)]$, recalling that $t_0$ is the unique root of the equation $P(t\Psi)=0$ and  $\dim_H X_{\omega}=t_0$ for $\mathbb{P}-$ a.e. $\omega\in\Omega$. 
\begin{corollary}\label{control Fh}
	For $\mathbb{P}$-a.e. $\omega$, for all $d\geq 0$,
	$$
	\dim_H \underline E(\nu_{\omega},d)\le \dim_H F(\omega,d)\leq d\cdot \sup_{\alpha> 0}\frac{\mathcal{T}^*(\alpha)}{\alpha}= d\cdot t_0.
	$$
	
\end{corollary}	
\begin{proof}
	For any $\varepsilon>0$, we saw that
	$$
	F(\omega,d)\subset \cup_{\alpha\in \mathbb{Q}^+}\cup_{\xi\in \mathbb{Q}\cap[1,+\infty),\alpha/\xi\leq d+2\varepsilon}G(\omega,\alpha,\varepsilon,\xi).$$
	Thus, from lemma \ref{control G} one has
	$$\dim_H F(\omega,d)\leq \sup_{\alpha\in \mathbb{Q}^+,\, \xi\in \mathbb{Q}\cap[1,+\infty):\alpha/\xi\leq d+2\varepsilon}C\varepsilon+\frac{\max (\mathcal{T}^\ast(\alpha-\varepsilon),\mathcal{T}^\ast(\alpha),\mathcal{T}^\ast(\alpha+\varepsilon))}{\xi}.$$
	Letting $\varepsilon$ tends to 0 yields
	$$\dim_H F(\omega,d)\leq \sup_{\alpha\ge 0 ,\xi\ge 1:\, \alpha/\xi\leq d}\frac{\mathcal{T}^*(\alpha)}{\xi}\leq d\cdot \sup_{\alpha> 0}\frac{\mathcal{T}^*(\alpha)}{\alpha}= d\cdot t_0.$$
	To getting the last equality, at first we notice that since $\mathcal{T}(t_0)=0$, we have  $\sup_{\alpha>0}\frac{\mathcal{T}^*(\alpha)}{\alpha}\geq \frac{\alpha t_0- \mathcal{T}(t_0)}{\alpha}=t_0$.
	Next, we know that for any $\alpha>0$, $\inf_{q}\{q\alpha-t_0\alpha-\mathcal{T}(q)\}\leq 0$, so  $\frac{\inf_{q}\{q\alpha-\mathcal{T}(q)\}}{\alpha}\leq  t_0,$
	which is $\frac{\mathcal{T}^*(\alpha)}{\alpha}\leq t_0$. Finally $\sup_{\alpha>0}\frac{\mathcal{T}^*(\alpha)}{\alpha}\leq t_0$.
	Now, recall~\eqref{inclusion}. Since the set of atoms of $\nu_\omega$ is countable, we get the desired conclusion for $\dim_H \underline E(\nu_\omega,d)$.
\end{proof}

\section{Lower bound for the $L^q$-spectrum and upper bound for the lower Hausdorff spectrum}
\begin{proposition}\label{proposition tau}
	For $\mathbb{P}$-a.e. $\omega\in\Omega$, for every $q\in \mathbb{R}$, we have $\tau_{\vw}(q)\geq\min(\mathcal{T}(q),0):=\widetilde{\mathcal{T}}(q)$.
\end{proposition}

Due to \eqref{MFI}, Proposition~\ref{proposition tau} gives the sharp upper bound for the lower Hausdorff spectrum.

Since  the functions $\tau_{\vw}$ and $\widetilde{\mathcal{T}}$ are both continuous, we just need to prove that the inequality of proposition~\ref{proposition tau}  holds on  a dense and countable subset of $\mathbb{R}$, which amounts to prove it for any fixed $q\in\mathbb{R}$, almost surely.

\begin{proof}
	
	
	Let $r>0$ and consider $\mathcal{B}=\{B_i\}$, a packing of $[0,1]$. that is a family of disjoint intervals $B_i$ with radius $r$ and centers in $[0,1]$.
	
	\medskip			
	\noindent
	{\bf Case 1:  $q\le 0$.} Set $B_i=:B(x_i,r)$. There exists a unique $v(x_i,r)\in \Sigma_{\omega,n}$ such that $x_i\in I_{\omega}^{v(x_i,r)}\subset B_i$ and $I_{\omega}^{(v(x_i,r))^\star} \not\subset B_i$.
	Here the notation $\star$ means that we delete the last character of the word.
	Then
	\begin{eqnarray*}
		2r &\geq& |I_{\omega}^{v(x_i,r)}|=\widetilde{\mu}_{\omega}([v(x_i,r)]_{\omega})\\
		&\geq& \exp(S_{n}\Phi(\omega,\vl)-n\epsilon(\Phi,\omega,n))) \geq \exp(-2nC_{\Phi}),
	\end{eqnarray*}
	for $r$ small enough. On the other hand, since $I_{\omega}^{(v(x_i,r))^\star} \not\subset B_i$, we have $$r\leq |I_{\omega}^{(v(x_i,r))^\star}|\leq \exp(S_{n-1}\Phi(\omega,\vl)+(n-1)\epsilon(\Phi,\omega,n-1)))\leq \exp(-\frac{(n-1)c_{\Phi}}{2})$$ (recall that $c_\phi=c_\Phi$ was defined in \eqref{cphi}. 
	Consequently, $\frac{\log 2r}{-2C_{\Phi}}\leq n\leq \frac{2\log r}{-c_{\Phi}}+1$.
	
	Also, $\nu_{\omega}(B_i)\geq\nu_{\omega}(I_{\omega}^{v(x_i,r)})\geq |X^{v(x_i,r)}_{\omega}|$. Since $q<0$, we get (using proposition~\ref{for n} applied to $q\Psi-\mathcal{T}(q)\Phi$ and its associated random weak Gibbs measure $\widetilde{\mu}^{q\Psi-\mathcal{T}(q)\Phi}_{\omega}$). Noticing $\mathcal{T}(q)$ for $q\leq 0$ and $\lim_{n\to\infty}\|\Phi(\sigma^{n-1}\omega))\|_{\infty}/n=0$, i.e. $\|\Phi(\sigma^{n-1}\omega))\|_{\infty}=o(n)$, we get
	\begin{eqnarray*}
		\vw(B_i)^q \leq |X^{v(x_i,r)}_{\omega}|^q
		&\leq& \exp(q(S_n\Psi(\omega,\vl)-n\epsilon(\Psi,\omega,n))))\quad(\forall \vl\in [v(x_i,r)]_{\omega} )\\
		&=&\exp((S_n(q\Psi-\mathcal{T}(q)\Phi)(\omega,\vl)+\mathcal{T}(q)S_{n}\Phi(\omega,\vl)-qn\epsilon(\Psi,\omega,n))\\
		&\leq&((S_n(q\Psi-\mathcal{T}(q)\Phi)(\omega,\vl)+\mathcal{T}(q)S_{n-1}\Phi(\omega,\vl)-qn\epsilon(\Psi,\omega,n)+o(n))\\
		&\leq& \widetilde{\mu}^{q\Psi-\mathcal{T}(q)\Phi}_{\omega}([v(x_i,r)]_{\omega})|I_{\omega}^{(v(x_i,r))^\star}|^{\mathcal{T}(q)}\exp(-qn\epsilon(\Psi,\omega,n)+o(n))\\
		&&\cdot\exp(n\epsilon(q\Psi-\mathcal{T}(q)\Phi,\omega,n)-\mathcal{T}(q)(n-1)\epsilon(\Phi,\omega,n-1))\\
		&\leq& \widetilde{\mu}^{q\Psi-\mathcal{T}(q)\Phi}_{\omega}([v(x_i,r)]_{\omega})r^{\mathcal{T}(q)}\exp(o(-\log r)).
	\end{eqnarray*}
	
	Thus $\sum_{B_i\in\mathcal{B}}\vw(B_i)^q\leq r^{\mathcal{T}(q)}\exp(o(-\log r))$, and
	letting $r\to 0$ yields $\tau_{\vw}(q)\geq \mathcal{T}(q)$.
	
	\medskip			
	\noindent{\bf Case 2:   $q\in (0,t_0)\subset (0,1)$.} Recall that $t_0=\dim_{H} X_{\omega}$ is the unique real number such that $P(t\Psi)=0$.
	Define
	\begin{align*}V(\omega,n,r)&=\{v\in\Sigma_{\omega,n}:|I_{\omega}^v|\geq 2r, \exists s \text{ such that } vs\in \Sigma_{\omega,n+1},\ |I_{\omega}^{vs}|< 2r \},\\
	V'(\omega,n,r)&=\{v\in V(\omega,n,r): \text{there is no }  k\in\mathbb{N} \text{ such that } v|_k\in V(\omega,k,r)\text{ with }k<n\},
	\end{align*}
	as well as 	$V(\omega,r)=\bigcup_{n\geq 1}V'(\omega,n,r), \ n_r=\max\{|v|:v\in V(\omega,r)\}, \text{ and }n'_r=\min\{|v|:v\in V(\omega,r)\}$.
	
	We have  $n_r=O(-\log r)=O(n'_r)$ and for any $v\in V(\omega,r)$ we have
	$$|I_{\omega}^v|\leq |I_{\omega}^{vs}|\exp(\|\Phi(\sigma^{|v|}\omega))\|_{\infty}+|v|\epsilon(\Phi,\omega,|v|)+(|v|+1)\epsilon(\Phi,\omega,|v|+1))),$$
	so that 
	\begin{equation}\label{control I-wv-r}
	|I_{\omega}^v|\leq 2r\exp(o(-\log r))
	\end{equation} 
	For $v\in V(\omega,r)$, $I_{\omega}^{v}$ meets at most $\exp(o(-\log r))$ intervals $B_i$ of the packing $\mathcal{B}$, and for every $B_i$ there are  at most two intervals $I_{\omega}^v$ and  $I_{\omega}^{v'}$ such that $B_i\subset I_{\omega}^v\cup I_{\omega}^{v'}$ and $v,v'\in V(\omega,r)$.
	Using the sub-additivity of the function $s\geq 0\mapsto s^q$, we get
	$$\nu_\omega(B_i)^q\leq \begin{cases}
	\nu_\omega({1})^{q}+\nu_\omega(I_{\omega}^v)^q+\nu_\omega(I_{\omega}^{v'})^q&\text{if }\ 1\in B_i ,\\
	\nu_\omega(I_{\omega}^v)^q+\nu_\omega(I_{\omega}^{v'})^q&\text{otherwise} \end{cases}.$$
	
	Recalling the definition of the inverse measure $\nu_\omega$ and proposition~\ref{definition nu} we know that $\nu_\omega(\widering{I}_{\omega}^v)\leq |X_{\omega}^v|$. Since  $0\leq q\leq t_0$, $\mathcal{T}(q)\leq 0$, then we get:
	\begin{eqnarray*}
		\nu_\omega(\widering{I}_{\omega}^v)^q\leq|X_{\omega}^v|^q
		&\leq& \exp(qS_{|v|}\Psi(\omega,\vl)+q{|v|} \epsilon(\Psi,\omega,{|v|}))\quad  (\forall \vl\in [v]_{\omega} )\\
		&\leq& \exp((S_{|v|}(q\Psi-\mathcal{T}(q)\Phi)(\omega,\vl)+\mathcal{T}(q)S_{{|v|}}\Phi(\omega,\vl)+q{|v|} \epsilon(\Psi,\omega,{|v|}))\\
		&\leq& \widetilde{\mu}^{q\Psi-\mathcal{T}(q)\Phi}_{\omega}([v]_{\omega})|I_{\omega}^{v}|^{\mathcal{T}(q)}\\
		&&\cdot\exp(q{|v|} \epsilon(\Psi,\omega,{|v|})+ {|v|} \epsilon(q\Psi-\mathcal{T}(q)\Phi,\omega,{|v|})-\mathcal{T}(q){|v|} \epsilon(\Phi,\omega,{|v|}))\\
		&\leq& \widetilde{\mu}^{q\Psi-\mathcal{T}(q)\Phi}_{\omega}([v]_{\omega})r^{\mathcal{T}(q)}\exp(o(-\log r)) (\text{ see }\eqref{control I-wv-r} \text{ and } {|v|}=O(-\log r)).
	\end{eqnarray*}
	It follows that
	\begin{eqnarray*}
		\sum_{B_i\in\mathcal{B}}\nu_\omega(B_i)^q &\leq & \nu_\omega({1})^{q}+\exp(o(-\log r))\Big (\sum_{n=n'_r}^{n_r}\sum_{v\in V(\omega,r)\cap\Sigma_{\omega,n}}\nu_\omega(\widering{I}_{\omega}^v)^q \\
		&& \quad\quad\quad\quad\quad\quad\quad\quad\quad\quad\quad\quad\quad\quad\quad\quad+\sum_{n=0}^{n_r}\sum_{v\in \Sigma_{\omega,n}}\sum_{s\in S'(\omega,v,1)}\nu_{\omega}(\{x_{\omega}^{vs}\})^q\Big ).
	\end{eqnarray*}
	Now, on the one hand,
	$$\sum_{n=n'_r}^{n_r}\sum_{v\in V(\omega,r)\cap\Sigma_{\omega,n}}\nu_\omega(\widering{I}_{\omega}^v)^q\leq r^{\mathcal{T}(q)}\exp(o(-\log r)),$$
	and on the other hand, 		
	for any $n\leq n_r$, with a similar way, we have that
	\begin{eqnarray*}
		&& \sum_{v\in \Sigma_{\omega,n}}\exp(qS_{n}\Psi(\omega,\vl)+o(n))(\forall \vl\in[v]_{\omega}) \\
		&\leq & \sum_{v\in \Sigma_{\omega,n}}\widetilde{\mu}^{q\Psi-\mathcal{T}(q)\Phi}_{\omega}([v]_{\omega}) r^{\mathcal{T}(q)}\exp(o(-\log r))
		\leq r^{\mathcal{T}(q)}\exp(o(-\log r)).
	\end{eqnarray*}
	Consequently,
	\begin{eqnarray*}
		&& \sum_{n=0}^{n_r}\sum_{v\in \Sigma_{\omega,n}}\sum_{s\in S'(\omega,v,1)}\nu_{\omega}(\{x_{\omega}^{vs}\})^q  \\
		&\leq& \sum_{n=0}^{n_r}\sum_{v\in \Sigma_{\omega,n}}\sum_{s\in S'(\omega,v,1)}(m_{\omega}^{vs'}-M_{\omega}^{vs})^q\\
		&\leq& \sum_{n=0}^{n_r}\sum_{v\in \Sigma_{\omega,n}}l(\sigma^n\omega)\exp(qS_{n}\Psi(\omega,\vl)+o(n))(\forall \vl\in[v]_{\omega})\\
		&\leq& r^{\mathcal{T}(q)}\exp(o(-\log r))(\text{noticing the fact that }\log n+\log l(\sigma^n\omega)=o(-\log r)).
	\end{eqnarray*}
	Now we obtain
	\begin{equation*}
	\sum_{B_i\in\mathcal{B}}\nu_{\omega}(B_i)^q\leq r^{\mathcal{T}(q)}\exp(o(-\log r)),
	\end{equation*}
	and letting $r\to 0$, we get $\tau_{\vw}(q)\geq \mathcal{T}(q)$.
	
	\medskip			
	\noindent
	{\bf Case 3: $q\geq t_0=\dim_{H} X_{\omega}$.} Since $\vw$ is discrete, we can easily get  $\tau_{\vw}(q)=0$ for every $q\geq 1$. For $q=t_0$, one has $\tau_{\vw}(q)\geq \mathcal{T}(q)=0$. Since the function $\tau_{\vw}$ is concave, we get $\tau_{\vw}(q)=0$ for every $q\geq t_0$.
\end{proof}

Next we collect 	information associated with the approximation of $(\Phi,\Psi)$ by pairs of H\"older continuous  potentials.

\section{Basic properties related to the approximation of $(\Phi,\Psi)$ by H\"older continuous random potentials}\label{section:Some preparation}
The material of this section, which can be skipped in a first reading, is borrowed from  \cite{Yuan2016} (here we just permute the roles of $\Phi$ and $\Psi$; this is natural since we work with the inverse measures of random weak Gibbs measures). Some facts will be collected in order to construct suitable measures supported on the level sets of the lower local dimension of $\nu_\omega$.

We fix two sequences $\{\Psi_i\}_{i\geq 1}$ and $\{\Phi_i\}_{i\geq 1}$ of random H\"older potentials as in \cite[section 3]{Yuan2016}, which converge to $\Phi$ and $\Psi$ respectively. Since  we assumed that $c_{\phi}=c_{\Phi}>0$, for each $i\in\N$ there exists a function $\mathcal{T}_i$ such that for any $q\in \mathbb{R}$ one has $P(q\Psi_i-\mathcal{T}_i(q)\Phi_i)=0$, and we have:

\begin{lemma}\label{converge TT}
	\begin{enumerate}
		\item  $\mathcal{T}_i$ converges poitwise to $\mathcal{T}$ as $i\to\infty$.
		\item $\mathcal{T}^*_i$ converges pointwise to $\mathcal{T}^*$ over the interior of the domain of $\mathcal{T}^*$ as $i\to\infty$.
	\end{enumerate}
\end{lemma}

Let $\mathcal{D}$ be a dense and countable subset of $(\mathcal{T}'(+\infty),\mathcal{T}'(-\infty))$, so that for any $d\in [\mathcal{T}'(+\infty),\mathcal{T}'(-\infty)]$, there exists $\{d_k\}_{k\in \mathbb{N}}\subset \mathcal{D}^{\mathbb{N}}$ such that $\lim_{k\to \infty} d_k = d$ and $\lim_{k\to \infty} \mathcal{T}^{*}(d_k)=\mathcal{T}^*(d)$.

Let $\{\mathcal{D}_i\}_{i\in \mathbb{N}}$ be a sequence of sets such that
\begin{itemize}
	\item $\mathcal{D}_i$ is a finite set for each $i\in \mathbb{N}$,
	\item $\mathcal{D}_i\subset \mathcal{D}_{i+1},$ for each $i\in \mathbb{N}$,
	\item $\cup_{i\in \mathbb{N}}\mathcal{D}_i=\mathcal{D}$.
\end{itemize}

Let us fix a positive sequence $\{\varepsilon_{i}\}_{i\in\mathbb{N}}$ decreasing to $0$. For each $i$, there exists $j_i$ large enough such that for any $d_i\in \mathcal{D}_i$, there exists $q_i\in \mathbb{R}$ such that
\begin{enumerate}
	\item $\mathcal{T}'_{j_i}(q_i)=d_i$,
	\item $|\mathcal{T}_{j_i}^*(d_i)-\mathcal{T}^*(d_i)|\leq (\varepsilon_i)^4$.
	\item $\int_{\Omega} \var_{j_i}\Psi\ d\mathbb{P}\le (\varepsilon_i)^4$ and
	$\int_{\Omega} \var_{j_i}\Phi\ d\mathbb{P}\le (\varepsilon_i)^4$
\end{enumerate}

Define $\mathcal{Q}_i=\{q_{i}, d_i\in \mathcal{D}_i\}$ and assume that $j_{i+1}>j_i$ for each $i\in\mathbb{N}$. For any $q\in Q_i$, we define
$$
\Lambda_{i,q}:=\widetilde{\Lambda}_{j_i,q}=q\Phi_{j_i}-\mathcal{T}_{j_i}(q)\Psi_{j_i}.
$$
With $\Lambda_{i,q}$ is associated a random Gibbs measure $\{\widetilde \mu^{\Lambda_{i,q}}_\omega\}_{\omega\in\Omega}$ and $\{\mu^{\Lambda_{i,q}}_\omega\}_{\omega\in\Omega}$ (see for instance \cite{Yuan2016} for the definition).
For any $n\in\mathbb{N}$, for any $v\in\Sigma_{\omega,n}$, define 
\begin{equation}\label{def-zeta}
\zeta^{{\Lambda}_{i,{q}}}_{\omega}(I_{\omega}^v)=\mu^{\Lambda_{i,q}}_\omega(U_{\omega}^v)=\widetilde \mu^{\Lambda_{i,q}}_\omega([v]_{\omega}).
\end{equation}

\begin{fact}\label{preparation}
	{\rm For each $i\geq 1$, for any $\epsilon>0$, there exist $ C>0$ (large enough), $M_i\in\mathbb{N}$ and a measurable set $\Omega(i)\subset \Omega$  such that the following basic properties (denoted by (BP)) hold for $\omega$, i.e.
		\begin{enumerate}
			\item $\mathbb{P}(\Omega(i))>3/4$;
			\item  $M(\omega)+1\leq M_{i}$;
			\item  for all $\omega\in\Omega(i)$, there is an increasing sequence $\{n^i_{k}=n^i_{k}(\omega)\}_{n\in\mathbb{N}}$ in $\mathbb{N}^\mathbb{N}$ such that $\sigma^{n^i_{k}+M_i}\omega\in \Omega(i)$
			\item for any $\omega\in\Omega(i)$, for $\Upsilon\in\{\Phi,\Psi\}$, $\forall n\in\mathbb{N}$, $\forall \vl\in\Sigma_{\omega}$, we have \begin{equation}\label{Control S-n upper}
			|S_{n}\Upsilon(\omega,\vl)|\leq nC;
			\end{equation}
			\item for any $\varepsilon_i>0$, for all $\omega\in\Omega(i)$,  there exists an integer $\mathscr{K}_i=\mathscr{K}_i(\omega)$ such that for any $d_i\in \mathcal D_i$, there exists $q_i\in \mathcal Q_i$ with $\mathcal{T}'_{j_i}({q_i})=d_i$  and $E_{i,q_i}=E_{i,q_i}(\sigma^{M_{i}}\omega)\subset [0,1]$ such that the following hold:
			\begin{itemize}
				\item the measure   $\tilde{\mu}_{\sigma^{M_{i}}\omega}^{\Lambda_{i,q_i}}$ is  well defined, and thus  so is $\zeta^{{\Lambda}_{i,{q_i}}}_{\sigma^{M_{i}}\omega}$ with $\zeta^{{\Lambda}_{i,{q_i}}}_{\sigma^{M_{i}}\omega}(E_{i,q_i})>1/2$;
				\item  
				\begin{equation}\label{control M}
				M_{i}\leq n^i_{\mathscr{K}_i}(\varepsilon_{i})^4;
				\end{equation}
				\item for $\Upsilon\in\{\Phi,\Psi\}$, for $p\geq n^i_{\mathscr{K}_i}$, we have 
				\begin{equation}\label{control Sp-var-ji}
				\left|\frac{S_p\mathrm{var}_{j_i}\Upsilon(\sigma^{M_{i}}\omega)}{p}-\int_{\Omega} \var_{j_i}\Upsilon d\mathbb{P}\right|\leq (\varepsilon_{i})^4,
				\end{equation}
				%
				%
				and
				$\forall \vl\in\Sigma_{\omega}$, we have
				\begin{equation}\label{control Upsilon small}
				S_{p}\Upsilon(F^{M_{i}}(\omega,\vl))\leq -\frac{pc_{\Upsilon}}{2};
				\end{equation}
				\item for any $k\geq \mathscr{K}_i$, 
				\begin{equation}\label{control nk}
				\frac{n^{i}_{k}-n^{i}_{k-1}}{n^{i}_{k-1}}\leq (\varepsilon_{i})^4;
				\end{equation}
				
				\item for any  $x\in {E}_{i,{q_i}}(\sigma^{M_{i}}\omega)$, for any $k\geq \mathscr{K}_i$, for any  $v\in \Sigma_{\sigma^{M_{i}}\omega,n^i_k}$ such that $x\in I_{\omega}^v$, one has
				$|v\wedge v+|\geq n^i_{k-1}$ and $|v\wedge v-|\geq n^i_{k-1}$.	
				Furthermore, for any  $\vl\in [v]_{\sigma^{M_{i}}\omega}\cup[v+]_{\sigma^{M_{i}}\omega}\cup[v-]_{\sigma^{M_{i}}\omega}$, one has
				\begin{equation}\label{control psi/Phi}
				\left|\frac{S_{n^i_k}\Psi_{j_i}(\sigma^{M_{i}}\omega,\vl)}{S_{n^i_k}\Phi_{j_i}(\sigma^{M_{i}}\omega,\vl)}-d_i\right|\leq (\varepsilon_i)^2,
				\end{equation}
				and for  $v'\in\{v, v+,v-\}$ and $\vl'\in [v']_{\sigma^{M_{i}}}$, one has 
				\begin{equation}\label{control zeta}
				\left|\frac{\log \zeta^{{\Lambda}_{i,{q_i}}}_{\sigma^{M_{i}}\omega}(I_{\sigma^{M_{i}}\omega}^{v'})}{S_{|v'|}\Phi_{j_i}(\sigma^{M_{i}}\omega,\underline{v'})}-
				\mathcal{T}^*(d_i)\right|\leq (\varepsilon_i)^2.
				\end{equation}
			\end{itemize}
		\end{enumerate}
	}
\end{fact}

\begin{fact}\label{modify2}
	{\rm
		We can change ${\Omega}(i)$ to ${\Omega}_i\subset {\Omega}(i)$ a bit smaller such that $\mathbb{P}(\Omega_i)\geq 1/2$ and	
		there exist two integers $\kappa_i$ and $W(i)$ such that for any $\omega\in {\Omega}_i$,
		$\mathscr{K}_i(\omega)\leq \kappa_i$, $n^{i}_{\kappa_i}(\omega)\leq W(i)$,
		and the properties listed in fact \ref{preparation} still hold.
	}	
\end{fact}

Then, we denote by ${\theta}(i,\omega,s)$ the $s$-th return time of the point $\omega$ to the set $ \Omega_{i}$, that is 
$${\theta}(i,\omega,1)=:\min\{n\in\mathbb{N}:\ \sigma^n\omega\in\Omega_i\},$$
$$\forall\ s\in\mathbb{N}, s>1,\ {\theta}(i,\omega,s)=:\min\{n\in\mathbb{N}:\ n>{\theta}(i,\omega,s-1) \text{ and } \sigma^n\omega\in\Omega_i\} $$

Since $\mathbb{N}$ is countable, there exists $\Omega'\subset\Omega$ of full $\mathbb P$-probability such that  for all $\omega\in\Omega'$,
$$\lim_{s\to\infty}\frac{{\theta}(i,\omega,s)}{s}=\frac{1}{\mathbb{P}(\Omega_i)}\text{ hence }
\lim_{s\to\infty}\frac{{\theta}(i,\omega,s)-{\theta}(i,\omega,s-1)}{ \theta(i,\omega,s-1)}=0.$$

By construction, for any $w\in \Sigma_{\omega,n}$ such that $\sigma^n \omega\in {\Omega}_i$, for any $d_i\in \mathcal{D}_i$, we have considered $q_i\in \mathcal{Q}_i$ such that $\mathcal{T}_{j_i}'(q_i)=d_i$. Remember the operation $\ast$ (see section~\ref{section:Setting and main results}),  which for $w\in \Sigma_{\omega,n}$ and $v\in\Sigma_{\sigma^{n+p}\omega,m}$ with $p\geq M(\sigma^{n}\omega)+1$, defines the word $w\ast v$.  Due to \eqref{def-zeta}, we can define a new measure $\zeta_{\omega,w,q_i}$ on  $ {I_{\omega}^{w}}$,  by setting, for  $v\in\Sigma_{\sigma^{n+M_{i}}\omega,*}$:
\begin{equation*}\label{}
\zeta_{\omega,w,q_i}(I_{\omega}^{w\ast v})=\zeta^{{\Lambda}_{i,{q_i}}}_{\sigma^{M_i}\omega}(I_{\sigma^{n+M_i}\omega}^{v})=\widetilde{\mu}_{\sigma^{n+M_i}\omega}^{{\Lambda}_{i,q_i}}([v]_{\sigma^{n+M_i}\omega}).
\end{equation*}
For any $y'\in I_{\sigma^{n+M_{i}}\omega}^v$, there exists $x'\in X_{\sigma^{n+M_{i}}\omega}^v$ such that $F_{\mu^{\phi}_{\sigma^{n+M_{i}}\omega}}(x')=y'$. Let $x=g_{\omega}^{w\ast}(x')$ and $y(\omega,w,y')=F_{\mu^{\phi}_{\omega}}(x)$. Also, define
$$
E(\omega,i,w,q_i)=\{y(\omega,w,y'): y'\in  E_{i,q_i}(\sigma^{n+{M_{i}}}\omega)\}.
$$
Then $E(\omega,i,w,q_i)\subset I_{\omega}^{w}$ and $\zeta_{\omega,w,q_i}(E(\omega,i,w,q_i))>1/2$.
\medskip

We will also use the following additional basic fact about the decay of $|I_{\omega}^v|$, which follows from proposition~\ref{for n}$(ii)$:  for all $\omega\in\Omega'$, for $p$ large enough (depending on $\omega$ only), for any $v\in\Sigma_{\omega,p}$, one has
\begin{equation}\label{control I-p}
\exp(-C p )\leq|I_{\omega}^v|\leq \exp(-\frac{c_{\Phi}p}{2}),
\end{equation}
with $C\geq 2C_{\Phi}>0$ (recall that $C_{\Phi}$ and $c_{\Phi}=c_{\phi}$ are defined in  \eqref{int} and  \eqref{cphi} respectively).

\section{A first lower bound for the lower Hausdorff spectrum}

In this section we establish the following lower bound for the lower Hausdorff spectrum.
\begin{proposition}\label{lower bound-1}
	For any $d\in [\mathcal{T}'(+\infty),\mathcal{T}'(-\infty)]$,
	$$\dim_H(\El(\nu_{\omega},d))\geq \dim_H(E(\nu_{\omega},d))\geq \mathcal{T}^*(d),$$
\end{proposition}
This bound will prove to be sharp for $d\in [\mathcal T'(t_0^-),\mathcal{T}'(-\infty)]$.
\begin{proof} The approach used in
	\cite{Yuan2016} also yields the validity of the multifractal formalism for $\mu_\omega$ if one considers the sets of the form $\{x\in X_{\omega}: \lim_{x\ni|I|\to 0}\frac{\log(\mu_{\omega}(I))}{\log(|I|)}=1/d\}$, like in \cite{RM1998}.
	Then using the general theory of \cite[theorem 21]{RM1998}, one gets 
	$$\dim_H\{x\in [0,1]: \lim_{x\ni|I|\to 0}\frac{\log(\nu_{\omega}(I))}{\log(|I|)}=d\}=\mathcal{T}^*(d).$$
	Since $\El(\nu_{\omega},d)\supset E(\nu_{\omega},d)\supset \{x\in [0,1]: \lim_{x\ni|I|\to 0}\frac{\log(\nu_{\omega}(I))}{\log(|I|)}=d\}$, 
	$$\dim_H(\El(\nu_{\omega},d))\geq \dim_H(E(\nu_{\omega},d))\geq \mathcal{T}^*(d).$$
\end{proof}
\begin{remark}
	It is possible to give a self-contained proof of the proposition based only on the basic properties of random weak Gibbs measures listed in the previous section. However, this is rather long and tedious, so we omit it. 
\end{remark}	

\section[Conditioned ubiquity]{Conditioned ubiquity}

In this section we extend to our context the conditioned ubiquity result obtained in~\cite{BS2007}. The statement takes the same form, but as explained in remark~\ref{rem1}(2), the proof must be revisited in order to cover the more general situation we face. The result will be applied to get the sharp lower bound for the lower Hausdorff spectrum on $ [0,\mathcal{T}'(t_0-)]$ (proposition~\ref{lower bound-2} in next section).

Recall remark~\ref{rem3}. We are interested in  the ubiquity of the family of points $\{z^v_\omega\}_{v\in\Sigma_{\omega,*}}$ relatively to the radii $\{\ell_{\omega}\}_{v\in\Sigma_{\omega,*}}$, and conditionally on the behavior of $\frac{S_{|v|}\Psi(\omega,\vl)}{S_{|v|}\Phi(\omega,\vl)}$ in~$[v]_\omega$.

\begin{definition} If $d\ge 0$, $\xi\geq 1$ and $\widetilde{\epsilon}=\{\epsilon_i\}_{i\in \N}$ is a positive sequence decreasing to~$0$ as $i\to \infty$, we set
	$$S(\omega,d,\xi,\widetilde{\epsilon}):=\bigcap_{N\geq 1}\bigcup_{n\geq N}\bigcup_{v\in\Sigma_{\omega,n},\exists \vl\in [v]_{\omega}\text{ such that }|\frac{S_{n}\Psi(\omega,\vl)}{S_{n}\Phi(\omega,\vl)}-d|\leq \epsilon_i }B(z_{\omega}^v,(\ell_{\omega}^v)^{\xi}).$$	
\end{definition}


\begin{theorem}\label{ubiqu}
	For $\mathbb{P}$-a.e. $\omega\in\Omega$, for any $\xi\geq 1$ and any exponent $d\in [\mathcal{T}'(+\infty),\mathcal{T}'(-\infty)]$, there exists a sequence $\widetilde{\epsilon}(\omega)=\{\epsilon_i(\omega)\}_{i\in \N}$ decreasing to  $0$ as $i\to \infty$, as well as a set $K^d(\xi)\subset S(\omega,d,\xi,\widetilde{\epsilon})$ and a Borel probability measure $m_{\xi}^{d}$ supported on $K^d(\xi)$  such that
	$\dim_H (m_{\xi}^d)\geq \frac{\mathcal{T}^*(d)}{\xi}$.
\end{theorem}

\begin{remark}
	In fact we can choose $\widetilde{\epsilon}$ independent of $\omega$. \end{remark}

\begin{remark}\label{S to El}
	For any $x\in S(\omega,d,\xi,\widetilde{\epsilon})$ there are infinitely many  $n_i$ and $v(i)\in \Sigma_{\omega,n_i}$ such that
	$x\in B(z_{\omega}^{v(i)},(\ell_{\omega}^{v(i)})^\xi)$ and $\alpha^{n_i}_{\omega}(x)\le d+\epsilon_i$, so
	\begin{equation*}\label{}
	\dl(\vw,x) \leq  \liminf_{i\to \infty}\frac{\log \nu_{\omega}(B(x,(\ell_{\omega}^{v(i)})^\xi)) }{\xi \log\ell_{\omega}^{v(i)}}\le  \liminf_{i\to \infty}\frac{\log \nu_{\omega}(\{z_{\omega}^{v(i)}\}) }{\xi \log\ell_{\omega}^{v(i)}}\le   \liminf_{i\to \infty}      \frac{ \alpha^{n_i}_{\omega}(x)}{\xi}\le  \frac{d}{\xi},
	\end{equation*}
	since by remark~\ref{rem3} $\frac{\log \nu_{\omega}(\{z_{\omega}^{v(i)}\})}{\log\ell_{\omega}^{v(i)}}$ is asymptotically not bigger than  the Birkhoff average $\frac{S_{|v_i|}\Psi(\omega,\vl)}{S_{|v_i|}\Phi(\omega,\vl)}$. Consequently, 	$S(\omega,d,\xi,\widetilde{\epsilon})\subset \bigcup_{h\leq \frac{d}{\xi}}\El(\vw,h).$	
\end{remark}

\begin{proof}[Proof of theorem~\ref{ubiqu}] Again, we will use the properties collected in section~\ref{section:Some preparation}. We can assume without loss of generality that for all $i\ge 1$, for all $d_i\in\mathcal D_i$, we have 
	\begin{equation}\label{condeps}
	\varepsilon_i\le \min\left \{\frac{c_{\Phi}}{16\xi(C+2)(C+1)+2c_{\Phi}},\frac{1}{4C},\frac{1}{3+\xi},\frac{c_{\Phi}}{14},\frac{c_{\Phi}}{20(d_{i}+1)} \text{ with } d_{i}\in \mathcal{D}_{i}\right \}
	\end{equation}
	(take $(\varepsilon_i)_{i\ge 1}$ which goes quickly to 0, while  $\min \{\mathcal T^*(d_i):d_i\in \mathcal D_i\}$ goes slowly to $\inf \,\mathcal T^*(\mathcal T'(\cdot))$).  
	Recall that ${\theta}(i,\omega,k)$ is the $k$-th return time of $\omega$ to the set $\Omega_{i}$ under the mapping $\sigma$. 
	
	We  start by constructing a generalized Cantor set $K(\xi,\widetilde{d})$ for each $\xi>1$ and each sequence $\widetilde{d}=(d_i)_{i=1}^\infty\in\prod_{i=1}^\infty\mathcal D_i$. 
	\begin{description}
		\item[Step 1] Let $\omega\in \Omega'$, choose  $n={\theta}(1,\omega,s)$ large enough such that:
		\begin{itemize}
			\item $\frac{\log 4\Gamma_1}{n}\leq 1,$  where $\Gamma_1$ is the constant in Besicovich's covering theorem on $\mathbb R$.
			\item for any $p\geq n$, for any $v\in\Sigma_{\omega,p}$, one has
			$\epsilon(\Phi,\omega,p)\leq (\varepsilon_{1})^4$ and \eqref{control I-p}.
			\item $W(1)\leq n(\varepsilon_{1})^4$.
			
		\end{itemize}
		
		Fix $w\in \Sigma_{\omega,n}$. Recall that for each $d_1\in \mathcal{D}_1$ there is ${q_1}\in \mathcal{Q}_1$ with $\mathcal{T}'_{j_1}({q_1})=d_1$.
		
		From fact~\ref{preparation} and fact~\ref{modify2}, we know that properties  (BP) hold for $\sigma^n\omega\in\Omega_{1}$. In this step, $\kappa_{1}$, $n^1_{k}$ are defined with respect to $\sigma^n\omega$,  and also $\zeta_{\omega,{w},{q_1}}$ and 
		$E(\omega,1,w,q_1))$ are well defined.
		
		For $k\ge \kappa_1$ and  $x\in E(\omega,1,w,q_1)$, let $v(\omega,1,q_1,n^1_{k},x)$ be the unique word such that $x\in I_{\omega}^{w\ast v(\omega,1,q_1,n^1_{k},x)}$ and $v(\omega,1,q_1,n^1_{k},x)\in \Sigma_{\sigma^{n+M_{1}}\omega,n^1_{k}}$. For any $k\geq \kappa_1$, define
		$$\mathcal{F}_1({q_1},n+M_1+n^1_{k})=\left \{B(y,2\ell_{\omega}^{w\ast v(\omega,1,q_1,n^1_{k},y)}):y\in E(\omega,1,w,q_1)\right \}.$$	
		Then $\mathcal{F}_1({q_1},n+n^1_{k})$ is a covering of $E(\omega,1,w,q_1)$. By Besicovitch's covering theorem \cite[theorem 2.7]{Mattila}, there are $\Gamma_1$ families of balls $\mathcal{F}_1^1({q_1},n+n^1_{k}),\cdots,\mathcal{F}_1^{\Gamma_1}({q_1},n+n^1_{k})\subset \mathcal{F}_1({q_1},n+n^1_{k})$, such that  $E(\omega,1,w,q_1)\subset \bigcup_{s=1}^{\Gamma_1}\bigcup_{B\in\mathcal{F}_1^s({q_1},n+n^1_{k})}B$ and for any $B,B'\in \mathcal{F}_1^s({q_1},n+n^1_{k})$, if $B\neq B'$ one has $B\cap B'=\emptyset$ (where $\Gamma_1$ is a constant  depending on the dimension $1$  of the Euclidean space~$\mathbb{R}$).
		
		Since $\zeta_{\omega,{w},{q_1}}(E(\omega,1,w,q_1))>1/2$, there exists $s$ such that
		\begin{equation*}\label{}
		\zeta_{\omega,{w},q_1}\Big(\bigcup_{B\in{\mathcal{F}_1^s({q_1},k)}}B\Big )\geq \frac{1}{2\Gamma_1}.
		\end{equation*}
		Among the intervals of $\mathcal{F}_1^s(k)$ we can choose a finite subset
		$$
		D^{w}(1,{d_1},k)=\{B_1,\cdots,B_{s'}\}
		$$
		such that
		\begin{equation*}\label{}
		\zeta_{\omega,{w},{q_1}}\Big (\bigcup_{B_l\in D^{w}(1,{d_1},k)}B_l\Big)\geq \frac{1}{4\Gamma_1}.
		\end{equation*}
		For any $B_l\in D^{w}(1,{d_1},k)$, there exists $y_l\in E(\omega,1,w,q_1)$ such that $B_l=B(y_l,2\ell_{\omega}^{w\ast v(k,l)})$, where $v(k,l):=v(\omega,1,q_1,n^1_{k},y_l)$.
		Since
		\begin{eqnarray*}
			B\big(z_{\omega}^{{w\ast v(k,l)}},(\ell_{\omega}^{w\ast v(k,l)})^{\xi}\big)&\subset& B(z_{\omega}^{w\ast v(k,l)},\ell_{\omega}^{w\ast v(k,l)})\\
			&\subset &B(y_l,2\ell_{\omega}^{w\ast v(k,l)})=B_l,
		\end{eqnarray*}
		Choose $\kappa'_{1}>\kappa_{1}$ large enough such that:
		\begin{itemize}
			\item  for any $v\in\Sigma_{\sigma^{n+M_{1}}\omega,n^{1}_{\kappa'_{1}}}$ one has
			$2\ell_{\omega}^{w\ast v}\leq |I_{\omega}^{w}|\exp(-n(\varepsilon_{1})^2)$;
			\item for any $j\geq n+M_1+n^{1}_{\kappa'_{1}}$, one has 
			$\epsilon(\Phi,\omega,j)\leq (\varepsilon_{2})^4$;
			\item $W({2})\leq (\varepsilon_{2})^4 (n+n^{1}_{\kappa'_{1}})$;
			\item for any $s$ such that the return time $\theta({2},\omega,s)$ satisfies $\theta({2},\omega,s)\geq n+M_1+n^{{1}}_{\kappa'_{1}}$,  one also has $$\frac{ \theta({2},\omega,s)- \theta({2},\omega,s-1)}{ \theta({2},\omega,s-1)}\leq (\varepsilon_{1})^4.$$	
		\end{itemize}
		
		For any $k\geq \kappa'_{1}$, we can get (this will be justified in the next step, see \eqref{control zeta B(y,r)}): 
		\begin{equation}\label{control  measure 1-1}
		\zeta_{\omega,w,q_{1}}\big(B(y_l,2\ell_{\omega}^{w\ast v(k,l)})\big)\leq \left(\frac{4\ell_{\omega}^{w\ast v(k,l)}}{|I_{\omega}^w|}\right)^{\mathcal{T}^*(d_{1})-\frac{\varepsilon_{1}}{2}}.
		\end{equation}
		
		Let $s_{2}=s_{2}(\omega,w)$ be the smallest $s$ such that there exists $v\in \Sigma_{\omega,\theta(2,\omega,s_{2})}$ such that
		\begin{itemize}
			\item  $z_{\omega}^{w\ast v(k,l)}$ belongs to the closure of the interval $I_{\omega}^{v}$,
			\item  $I_{\omega}^{v}\subset B(z_{\omega}^{w\ast v(k,l)},(\ell_{\omega}^{w\ast v(k,l)})^{\xi})$.
		\end{itemize}
		By definition of $s_{2}$, setting $v'=v|_{\theta(2,\omega,s_{2}-1)}\in\Sigma_{\omega,\theta(2,\omega,s_{2}-1)}$, we have $|I_{\omega}^{v'}|\geq (\ell_{\omega}^{w\ast v(k,l)})^{\xi}$.
		Now let $\mathcal{K}_{1}$ be the largest $k$ such that $n+M_1+n^{1}_{k}\leq \theta(2,\omega,s_{2}-1)$ (by construction we have $\mathcal K_{1}\ge \kappa'_{1}$). Due to \eqref{control nk}, we have 
		$$\theta(2,\omega,s_{2}-1)-n-M_1-n^{1}_{\mathcal{K}_{1}}\leq  n^{1}_{\mathcal{K}_{1}+1}-n^{1}_{\mathcal{K}_{1}}\le (\varepsilon_{1})^4 n^{1}_{\mathcal{K}_{1}}.$$
		Then (this will be justified in the next step in \eqref{control overline{J_l}-{xi}})
		$$
		(2\ell_{\omega}^{w\ast v(k,l)})^{\xi})\leq 2^{\xi}|I_{\omega}^{v'}|
		\leq 2^{\xi}|I_{\omega}^{v}|^{1-(\varepsilon_{1})^3}\leq |I_{\omega}^{v}|^{1-(\varepsilon_{1})^2}.$$

		Define $J_l$ to be $\overline{I_v}$,  the closure of the interval $I_{\omega}^{v}$, and denote $\underline{B_l}=J_l$ and $B_l=\widehat{J_l}$. We get: 
		\begin{equation}\label{control Jl-1}
		|J_l|\leq |\widehat{J_l}|^{\xi}\leq |J_l|^{1-(\varepsilon_{1})^2}.
		\end{equation}
		
		Now using lemma \ref{lemma converge to d} of the next step, we can get that for $k$ large enough so that $n^{1}_{k}\geq\frac{n}{(\varepsilon_{1})^2}$, then for any $\vl\in [w\ast v(k,l)]_{\omega}$ we have:
		$$\left|\frac{S_{n+n^{1}_{k}}\Psi(\omega,\vl)}{S_{n+n^{1}_{k}}\Phi(\omega,\vl)}-d\right|\leq \epsilon_{1}.$$

		For $k>\kappa'_{1}$ large enough with $n^{1}_{k}\geq\frac{n}{(\varepsilon_{1})^2}$, define
		\begin{equation*}
		G^w(1,d_{1},k)=\{\underline{B_l},B_l\in D^w(1,d_{1},k)\}.
		\end{equation*}
		If $J_1$ and $J_2$ are two distinct elements of $G^w(1,d_{1},k)$ then their distance is at least $\max_{j\in \{1,2\}}(|\widehat{J_j}|/2-(|\widehat{J_j}|/2)^{\xi})$, which is larger than $\max_{j\in \{1,2\}}|\widehat{J_j}|/3$ for $k$ large enough (since $\xi>1$).
		
		Now we can define $m_{\xi}^{\{d_1\}}$ with $d_{1}\in\mathcal{D}_{1}$ as follows:
		\begin{equation*}\label{}
		m_{\xi}^{\{d_1\}}(J)=\frac{\zeta_{\omega,w,q_{1}}(\widehat{J})}{\sum_{J_l\in G^w(1,d_{1},k)}\zeta_{\omega,w,q_{1}}(\widehat{J_l})}.
		\end{equation*}
		For any $J\in  G^w(1,d_{1},k)$, by inequality \eqref{control  measure 1-1} and \eqref{control Jl-1}, we can get
		$$\zeta_{\omega,w,q_{1}}(\widehat{J})\leq |J|^{\frac{\mathcal{T}^*(d_{1})-\frac{2\varepsilon_{1}}{3}}{\xi}}|I_{\omega}^w|^{-\mathcal{T}^*(d_{1})}.$$
		Then, the inequality
		$
		\sum_{J_l\in G^w(1,d_{1},k)}\zeta_{\omega,w,q_{1}}(\widehat{J_l})\geq \frac{1}{4\Gamma_1}
		$
		yields, $\forall J\in G^w(1,d_{1},k)$:
		$$m_{\xi}^{\{d_1\}}(J)\leq 4\Gamma_1 |J|^{\frac{\mathcal{T}^*(d_{1})-\frac{2\varepsilon_{1}}{3}}{\xi}}|I_{\omega}^w|^{-\mathcal{T}^{*}(d_{1})}.$$
		
		Choose $k_1>\mathcal{N}'_{1}$ large enough with $n^1_{k_1}\geq\frac{n}{(\varepsilon_{1})^2}$. Then, for any $d_{1}\in\mathcal D_{1}$ and any $J\in G^w(1,d_{1}, k_{1})$, one has
		$4\Gamma_1|I_{\omega}^w|^{-\mathcal{T}^{*}(d_{1})}\leq |J|^{-\frac{\varepsilon_{1}}{3\xi}}$  (this will follow from a general estimate in the next step).
		
		Finally,  		   
		\begin{equation*}
		\forall J\in G^w(1,d_{1}, k_{1}),\ m_{\xi}^{\{d_1\}}(J)\leq  |J|^{\frac{\mathcal{T}^*(d_{1})-\varepsilon_{1}}{\xi}}.
		\end{equation*}
		For any $d_1\in\mathcal{D}_1$ define:
		$$G(d_1)=G^w(1,d_{1}, k_{1}),$$
		and
		$$G_{1}=\bigcup_{d_{1}\in \mathcal{D}_{1}}G^w(1,d_{1}, k_{1}).$$
		
		\item[Step 2]  Suppose that $G_i=\bigcup_{(d_1,\ldots,d_i)\in\prod_{j=1}^i\mathcal D_j}G(d_1,\ldots,d_i)$ is well defined and for any $\{d_j\}_{1\leq j\leq i}\in \prod_{1\leq j\leq i}\mathcal{D}_j$, the set function $m^{\{d_j\}_{1\leq j\leq i}}_{\xi}$ is well defined on the set $G(d_1,\ldots, d_i)$.
		
		For any $w $ such that $J$, the closure of $I^w_\omega$, belongs to $G(d_1\cdots d_i)\subset G_i$, we set $n=|w|$. In this step $n^{i+1}_k$ stands for $n^{i+1}_k(\sigma^n\omega)$.

		By construction:
		\begin{enumerate}
			\item $\sigma^n\omega\in {\Omega}_{i+1}$,
			\item for $p\geq n$, we have 
			\begin{equation}\label{control epsilon(Phi,omega,p)}
			\epsilon(\Phi,\omega,p)\leq (\varepsilon_{i+1})^4,
			\end{equation}
			\item $W(i+1)\leq n(\varepsilon_{i+1})^4$,
		\end{enumerate}
		For any $d_{i+1}\in \mathcal{D}_{i+1}$, take $q_{i+1}\in \mathcal{Q}_{i+1}$ such that $\mathcal{T}_{j_{i+1}}'(q_{i+1})=d_{i+1}$. From fact~\ref{preparation} and fact~\ref{modify2}, we know that (BP) hold for $\sigma^n\omega\in\Omega_{i+1}$. Also $\zeta_{\omega,{w},{q_{i+1}}}$ and $E(\omega,i+1,w,q_{i+1}))$ are well defined.
		
		For any $k\geq \kappa_{i+1}$, let
		$$\mathcal{F}_{i+1}(q_{i+1},n+M_{i+1}+n^{i+1}_k)=\{B(y,2\ell_{\omega}^{w\ast v(\omega,i+1,q_{i+1},n^{i+1}_k,y)}):y\in E(\omega,i+1,w,q_{i+1})\},$$
		where $v(\omega,i+1,q_{i+1},n^{i+1}_k,y)$ is the unique word such that $y\in I_{\omega}^{w\ast v(\omega,i+1,q_{i+1},n^{i+1}_k,y)}$ and $v(\omega,i+1,q_{i+1},n^{i+1}_k,y)\in\Sigma_{\sigma^{n+M_{i+1}}\omega,n^{i+1}_{k}}$.
		Then $\mathcal{F}_{i+1}(q_{i+1},n+M_{i+1}+n^{{i+1}}_k)$ is a covering of $E(\omega,i+1,w,q_{i+1})$.
		
		From the Besicovitch's covering theorem, $\Gamma_1$ families of disjoint balls, namely
		$$\mathcal{F}_{i+1}^1(q_{i+1},n+M_{i+1}+n^{{i+1}}_k),\cdots,\mathcal{F}_{i+1}^{\Gamma_1}(n+M_{i+1}+n^{{i+1}}_k)$$
		can be extracted from $\mathcal{F}_{i+1}(q_{i+1},n+M_{i+1}+n^{i+1}_k)$ so that 
		$$
		E(\omega,i+1,w,q_{i+1})\subset \bigcup_{s=1}^{\Gamma_1} \bigcup_{B\in \mathcal{F}_{i+1}^s(q_{i+1},n+M_{i+1}+n^{{i+1}}_k)}B.
		$$ Since $\zeta_{\omega,w,q_{i+1}}(E(\omega,i+1,w,q_{i+1}))\geq 1/2$, there exists $s$ such that
		\begin{equation*}\label{}
		\zeta_{\omega,w,q_{i+1}}\Big(\bigcup_{B\in{\mathcal{F}_{i+1}^s(q_{i+1},n+M_{i+1}+n^{{i+1}}_k)}}B\Big)\geq \frac{1}{2\Gamma_1}.
		\end{equation*}
		
		Again, we extract from $\mathcal{F}_{i+1}^s(q_{i+1},n+M_{i+1}+n^{{i+1}}_k)$ a finite family of pairwise disjoint intervals $D^w(i+1,d_{i+1},k)=\{B_1,\cdots,B_{s'}\}$ such that
		\begin{equation*}\label{}
		\zeta_{\omega,w,q_{i+1}}\Big(\bigcup_{B_l\in D^w(i+1,d_{i+1},k)}B_l\Big)\geq \frac{1}{4\Gamma_1}.
		\end{equation*}
		For each $B_l\in D^w(i+1,q_{i+1},k)$, there exists $y_l\in E(\omega,i+1,w,q_{i+1})$ such that $B_l=B\big(y_l,2\ell_{\omega}^{w\ast v(\omega,i+1,q_{i+1},n^{{i+1}}_k,y_l)}\big)$. Set $v(k,l)=v(\omega,i+1,q_{i+1},n^{{i+1}}_k,y_l)$. Moreover,
		\begin{eqnarray*}
			B\big(z_{\omega}^{w\ast v(k,l)},(\ell_{\omega}^{w\ast v(k,l)})^{\xi}\big) &\subset& B\big(z_{\omega}^{w\ast v(k,l)},\ell_{\omega}^{w\ast v(k,l)}\big) \\
			&\subset & B(y_l,2\ell_{\omega}^{w\ast v(k,l)})=B_l.
		\end{eqnarray*}	
		We can get the following lemma 		
		\begin{lemma}\label{lemma control zeta B(y,r)}
			For any $y\in E(\omega,i+1,w,q_{i+1})$, for $r\leq |I_{\omega}^{w}|\exp(-n(\varepsilon_{i+1})^2)$, we have 
			\begin{equation}\label{control zeta B(y,r)}
			\zeta_{\omega,w,q_{i+1}}\big(B(y,r)\big)\leq \left(\frac{r}{|I_{\omega}^{w}|}\right)^{\mathcal{T}^*(d_{i+1})-\frac{\varepsilon_{i+1}}{2}}.
			\end{equation}
		\end{lemma}	
		\begin{proof}[Proof of the lemma.]
			The idea of the proof is the following. Noticing $y\in E(\omega,i+1,w,q_{i+1})$, If $r$ small enough, we will know that $v(n+M_{i+1}+n^{i+1}_{k},y)$ is a comment prefix $v(n+M_{i+1}+n^{i+1}_{k+1},y)$ and its left and right. This will implies $B(x,r)\subset v(n+M_{i+1}+n^{i+1}_{k+1},y)$ and $r$ is comparable with $|v(n+M_{i+1}+n^{i+1}_{k+1},y)|$. The result follows since we have a good control of the measure $v(n+M_{i+1}+n^{i+1}_{k},y)$.
			
			First, for any $y\in E(\omega,i+1,w,q_{i+1})$, for any $k\geq \kappa_{i+1}$ such that $y\in I_{\omega}^{w\ast v}$ with $v\in\Sigma_{\sigma^{n+M_{i+1}}\omega,n^{i+1}_k}$, for any $v'\in\Sigma_{\sigma^{n+M_{i+1}}\omega,n^{i+1}_{k+1}}$ with $v'|_{n^{i+1}_k}=v$, we have
			\begin{eqnarray}
			\nonumber          \frac{|I_{\omega}^{w\ast v'}|}{|I_{\omega}^{w\ast v}|}&=&\frac{\widetilde{\mu}_{\omega}([w\ast v']_{\omega})}{\widetilde{\mu}_{\omega}([w\ast v]_{\omega})}\\
			\nonumber      &\geq&\exp\big(S_{n^{i+1}_{k+1}-n^{i+1}_{k}}\Phi(F^{n+M_{i+1}+n^{i+1}_{k}}(\omega,\vl'))-2(n+M_{i+1}+n^{i+1}_{k+1})(\varepsilon_{i+1})^4\big)\ (\text{see }\eqref{control epsilon(Phi,omega,p)})\\
			\nonumber          &\geq&\exp\big(-(n^{i+1}_{k+1}-n^{i+1}_{k})C-2(n+M_{i+1}+n^{i+1}_{k+1})(\varepsilon_{i+1})^4\big)\\
			\nonumber &&(\text{noting } \sigma^{n+M_{i+1}+n^{i+1}_{k}}\omega\in\Omega_{i+1} \text{ and } \eqref{Control S-n upper})\\
			\label{ineg0}           &\geq&\exp\big(-(C+2)(n+M_{i+1}+n^{i+1}_{k})(\varepsilon_{i+1})^4\big)\ (\text{see }\eqref{control nk}).
			\end{eqnarray}
			Next, choose the largest $k$ and $v\in\Sigma_{\sigma^{n+M_{i+1}}\omega,n^{i+1}_k}$ such that $y\in I_{\omega}^{w\ast v}$ and $r\leq |I_{\omega}^{w\ast v}|\exp(-(C+2)(n+M_{i+1}+n^{i+1}_{k})(\varepsilon_{i+1})^4)$. 
			We have, applying the shorthand $\widetilde F^n=F^{n+M_{i+1}}$:
			\begin{eqnarray*}\label{}
				\frac{r}{|I_{\omega}^{w}|}&\leq&\frac{|I_{\omega}^{w\ast v}|\exp\big(-(C+2)(n+M_{i+1}+n^{i+1}_{k})(\varepsilon_{i+1})^4\big)}{|I_{\omega}^{w}|}\\&=&\frac{\widetilde{\mu}_{\omega}([w\ast v]_{\omega})\exp\big(-(C+2)(n+M_{i+1}+n^{i+1}_{k})(\varepsilon_{i+1})^4\big)}{\widetilde{\mu}_{\omega}([w]_{\omega})}\\
				&\leq&\frac{\widetilde{\mu}_{\omega}([w\ast v]_{\omega})\exp(-(C+2)(n+M_{i+1}+n^{i+1}_{k})(\varepsilon_{i+1})^4)}{\widetilde{\mu}_{\omega}([w\ast]_{\omega})}\\
				&\leq&\exp\big(S_{n^{i+1}_{k}}\Phi(\widetilde F^n(\omega,\vl))+(n+M_{i+1})\epsilon(\Phi,\omega,n+M_{i+1})\big)\\
				&&\cdot\exp\big((n+n^{i+1}_{k})\epsilon(\Phi,\omega,n+M_{i+1}+n^{i+1}_{k})-(C+2)(n+M_{i+1}+n^{i+1}_{k})(\varepsilon_{i+1})^4\big)\\
				&\leq&\exp(S_{n^{i+1}_{k}}\Phi(\widetilde F^n(\omega,\vl)))\\
				&&\cdot \exp \big(2(n+M_{i+1}+n^{i+1}_{k})(\varepsilon_{i+1})^4-(C+2)(n+M_{i+1}+n^{i+1}_{k})(\varepsilon_{i+1})^4\big)\ (\text{due to \eqref{control epsilon(Phi,omega,p)}})\\
				&&\\
				&\leq&\exp\Big (-\frac{(n^{i+1}_{k})c_{\Phi}}{2}-C(n+M_{i+1}+n^{i+1}_{k})(\varepsilon_{i+1})^4\big )\ (\text{due to  \eqref{control Upsilon small}})\\
				&\leq& \exp \Big (-\frac{n^{i+1}_{k}c_{\Phi}}{2} \Big ) \ (\text{due to  \eqref{control M}}),
			\end{eqnarray*} 
			so 
			\begin{equation}\label{control n^{i+1}_{k}}
			n^{i+1}_{k}\leq-\frac{2\log (\frac{r}{|I_{\omega}^{w}|})}{c_{\Phi}}.
			\end{equation}      
			Also, since $y\in E(\omega,i+1,w,q_{i+1})$, for $v'\in\Sigma_{\sigma^{n+M_{i+1}}\omega,n^{i+1}_{k+1}}$ such that $y\in I_{\omega}^{w\ast v'}$, we know that $v$ is a common prefix of $v'+, v'-$ and $v'$. Thus, due to \eqref{ineg0} and our choice for $k$,  $B(y,r)\subset (I_{\omega}^{w\ast v'-}\cup I_{\omega}^{w\ast v'}\cup I_{\omega}^{w\ast v'+})\subset I_{\omega}^{w\ast v}$ and
			$r\geq |I_{\omega}^{w\ast v'}|\exp(-(C+2)(n+M_{i+1}+n^{i+1}_{k+1})(\varepsilon_{i+1})^4)$. From the last inequality we can get
			\begin{eqnarray*}\label{}
				\frac{r}{|I_{\omega}^{w}|}&\geq&\frac{|I_{\omega}^{w\ast v'}|\exp(-(C+2)(n+M_{i+1}+n^{i+1}_{k+1})(\varepsilon_{i+1})^4)}{|I_{\omega}^{w}|}\\
				&=&\frac{\widetilde{\mu}_{\omega}([w\ast v']_{\omega})\exp(-(C+2)(n+M_{i+1}+n^{i+1}_{k+1})(\varepsilon_{i+1})^4)}{\widetilde{\mu}_{\omega}([w]_{\omega})}\\
				&\geq&\exp\big (-(C+2)(n+M_{i+1}+n^{i+1}_{k+1})(\varepsilon_{i+1})^4\big )\\
				&&\cdot\exp\big (S_{M_{i+1}+n^{i+1}_{k+1}}\Phi(F^{n}(\omega,\vl'))-2(n+M_{i+1}+n^{i+1}_{k+1})(\varepsilon_{i+1})^4\big)\\
				&&\text{(due to proposition~\ref{for n} and \eqref{control epsilon(Phi,omega,p)} )}\\
				&\geq& \exp\big (S_{n^{i+1}_{k}}\Phi(\widetilde F^n(\omega,\vl'))-M_{i+1}C-(n^{i+1}_{k+1}-n^{i+1}_{k})C\big )\\
				&&\quad\quad\cdot\exp\big (-(C+4)(n+M_{i+1}+n^{i+1}_{k+1})(\varepsilon_{i+1})^4\big )\\
				&\geq&\exp\big (S_{n^{i+1}_{k}}\Phi(\widetilde F^n(\omega,\vl'))-2n^{i+1}_{k}C(\varepsilon_{i+1})^4\big )\\
				&&\quad\quad\cdot\exp\big (-(C+4)(n+n^{i+1}_{k}(\varepsilon_{i+1})^4+n^{i+1}_{k}(1+(\varepsilon_{i+1})^4))(\varepsilon_{i+1})^4\big )\\
				&&\text{(see \eqref{control M}to control $M_{i+1}$ and \eqref{control nk} for $n^{i+1}_{k+1}-n^{i+1}_{k}$ since $k\geq \kappa_{i+1}$).}
			\end{eqnarray*}
			Thus, noticing that $\varepsilon_{i+1}<1/2$, we get  
			$$       \frac{r}{|I_{\omega}^{w}|}   \geq\exp\big (S_{n^{i+1}_{k}-M_{i+1}}\Phi(\widetilde F^n(\omega,\vl'))-(4C+8)(n+n^{i+1}_{k})(\varepsilon_{i+1})^4\big ).
			$$
			Since $r\leq |I_{\omega}^{w}|\exp(-n(\varepsilon_{i+1})^2)$, we deduce 
			\begin{eqnarray*}
				n(\varepsilon_{i+1})^2&\leq& -S_{n^{i+1}_{k}}\Phi(\widetilde F^n(\omega,\vl')+(4C+8)(n+n^{i+1}_{k})(\varepsilon_{i+1})^4\\
				&\leq& Cn^{i+1}_{k}+(4C+8)(n+n^{i+1}_{k})(\varepsilon_{i+1})^4.
			\end{eqnarray*}
			This yields $n^{i+1}_{k}\geq \frac{n(\varepsilon_{i+1})^2-(4C+8)n(\varepsilon_{i+1})^4}{C+(4C+8)(\varepsilon_{i+1})^4}\geq\frac{n(\varepsilon_{i+1})^2}{2C}$, so $n\leq 2Cn^{i+1}_{k}/((\varepsilon_{i+1})^2)$, then,
			\begin{equation}\label{control r/I lower}
			\frac{r}{|I_{\omega}^{w}|}\geq \exp\big (S_{n^{i+1}_{k}}\Phi(\widetilde F^n(\omega,\vl')-(4C+8)(2C+1)n^{i+1}_{k}(\varepsilon_{i+1})^2\big). 
			\end{equation} 
			
			Consequently, 
			\begin{eqnarray*}
				&&\zeta_{\omega,w,q_{i+1}}(B(y,r))\\
				&\leq&\zeta_{\omega,w,q_{i+1}}(I_{\omega}^{w\ast v})=\zeta^{{\Lambda}_{{i+1},{q_{i+1}}}}_{\sigma^{n+M_{i+1}}\omega}(I_{\sigma^{n+M_{i+1}}\omega}^{v})\\
				&\leq&\exp((\mathcal{T}^*(d_{i+1})-(\varepsilon_{i+1})^2)S_{n^{i+1}_{k}}\Phi_{j_{i+1}}(\widetilde F^n(\omega,\vl')))\ (\text{ see \eqref{control zeta}})\\
				&\leq& \exp\big((\mathcal{T}^*(d_{i+1})-(\varepsilon_{i+1})^2)(S_{n^{i+1}_{k}}\Phi(\widetilde F^n(\omega,\vl'))+2n^{i+1}_{k}(\varepsilon_{i+1})^4)\big)\\
				&\leq&\left(\frac{r}{|I_{\omega}^{w}|}\right)^{\mathcal{T}^*(d_{i+1})-(\varepsilon_{i+1})^2}\\
				&&\quad \quad \cdot\exp\big((\mathcal{T}^*(d_{i+1})-(\varepsilon_{i+1})^2)((4C+8)(2C+1)n^{i+1}_{k}(\varepsilon_{i+1})^2)+2n^{i+1}_{k}(\varepsilon_{i+1})^4\big),
			\end{eqnarray*} where we have used \eqref{control r/I lower}.  Then, 
			\begin{eqnarray*}
				\zeta_{\omega,w,q_{i+1}}(B(y,r))&\leq&\left(\frac{r}{|I_{\omega}^{w}|}\right)^{\mathcal{T}^*(d_{i+1})-(\varepsilon_{i+1})^2}\exp\big(8(C+2)(C+1)n^{i+1}_{k}(\varepsilon_{i+1})^2\big)\\
				&\leq&\left(\frac{r}{|I_{\omega}^{w}|}\right)^{\mathcal{T}^*(d_{i+1})-(\varepsilon_{i+1})^2}	\left(\frac{r}{|I_{\omega}^{w}|}\right )^{-\frac{8(C+2)(C+1) (\varepsilon_{i+1})^2}{2c_{\Phi}}}\  {\rm ( using \ \eqref{control n^{i+1}_{k}})}\\
				&\leq&\left(\frac{r}{|I_{\omega}^{w}|}\right)^{\mathcal{T}^*(d_{i+1})-(\frac{4(C+2)(C+1)}{c_{\Phi}}+1)(\varepsilon_{i+1})^2}\leq\left(\frac{r}{|I_{\omega}^{w}|}\right)^{\mathcal{T}^*(d_{i+1})-\frac{\varepsilon_{i+1}}{2}}
			\end{eqnarray*}  
			(\text{from the choice }$\varepsilon_{i+1}\leq \frac{c_{\Phi}}{16(C+2)(C+1)+2c_{\Phi}}$).  \end{proof}
		Choose $\kappa'_{i+1}>\kappa_{i+1}$ large enough so that:
		\begin{itemize}
			\item  for any $v\in\Sigma_{\sigma^{n+M_{i+1}}\omega,n^{i+1}_{\kappa'_{i+1}}}$ one has
			$2\ell_{\omega}^{w\ast v}\leq |I_{\omega}^{w}|\exp(-n(\varepsilon_{i+1})^2)$;
			\item for any $j\geq n+n^{i+1}_{\kappa'_{i+1}}$, one has 
			$\epsilon(\Phi,\omega,j)\leq (\varepsilon_{i+2})^4$;
			\item $W({i+2})\leq (\varepsilon_{i+2})^4 (n+M_{i+2}+n^{i+1}_{\kappa'_{i+1}})$;	
			\item for any $s$ such that the return time $\theta({i+2},\omega,s)$ satisfies $\theta({i+2},\omega,s)\geq n+M_{i+1}+n^{{i+1}}_{\kappa'_{i+1}}$,  one also has $$\frac{ \theta({i+2},\omega,s)- \theta({i+2},\omega,s-1)}{ \theta({i+2},\omega,s-1)}\leq (\varepsilon_{i+1})^4.$$	
		\end{itemize}
		
		For any $k\geq \kappa'_{i+1}$, from \eqref{control zeta B(y,r)} we can get: 
		\begin{equation}\label{control  measure i-1}
		\zeta_{\omega,w,q_{i+1}}(B(y_l,2\ell_{\omega}^{w\ast v(k,l)}))\leq \left(\frac{4\ell_{\omega}^{w\ast v(k,l)}}{|I_{\omega}^w|}\right)^{\mathcal{T}^*(d_{i+1})-\frac{\varepsilon_{i+1}}{2}}.
		\end{equation}
		
		Let $s_{i+2}=s_{i+2}(\omega,w)$ be the smallest $s$ such that there exists $v\in \Sigma_{\omega,\theta(i+2,\omega,s_{i+2})}$ such that
		\begin{itemize}
			\item  $z_{\omega}^{w\ast v(k,l)}$ belongs to the closure of the interval $I_{\omega}^{v}$,
			\item  $I_{\omega}^{v}\subset B(z_{\omega}^{w\ast v(k,l)},(\ell_{\omega}^{w\ast v(k,l)})^{\xi})$
		\end{itemize}
		Define $J_l$ to be $\overline{I_{\omega}^{v}}$, the closure of the interval $I_{\omega}^{v}$, and denote $\underline{B_l}=J_l$ and $B_l=\widehat{J_l}$. From the construction,
		we can claim: 
		\begin{equation}\label{control overline{J_l}-{xi}}
		|J_l|\leq |\widehat{J_l}|^{\xi}\leq |J_l|^{1-(\varepsilon_{i+1})^2}.
		\end{equation}
		
		Since $s_{i+2}$ is the smallest one, so for $v'=v|_{\theta(i+2,\omega,s_{i+2}-1)}\in\Sigma_{\omega,\theta(i+2,\omega,s_{i+2}-1)}$, we have $|I_{\omega}^{v'}|\geq (\ell_{\omega}^{w\ast v(k,l)})^{\xi}$.
		Now let $\mathcal{K}_{i+1}$ be the largest $k$ such that $n+M_{i+1}+n^{i+1}_{k}\leq \theta(i+2,\omega,s_{i+2}-1)$ (by construction we have $k\ge \kappa'_{i+1}$). Due to \eqref{control nk}, we have 
		$$\theta(i+2,\omega,s_{i+2}-1)-n-M_{i+1}-n^{i+1}_{\mathcal{K}_{i+1}}\leq  n^{i+1}_{\mathcal{K}_{i+1}+1}-n^{i+1}_{\mathcal{K}_{i+1}}\le (\varepsilon_{i+1})^4 n^{i+1}_{\mathcal{K}_{i+1}}.$$
		
		Using a similar method as in the proof of \eqref{control nk}, we can get 
		\begin{equation*}\label{Control I-v-v'}
		\frac{|I_{\omega}^{v'}|}{|I_{\omega}^{v}|}\leq \exp\big(2(C+1)\theta(i+2,\omega,s_{i+2})(\varepsilon_{i+1})^4)\big)
		\end{equation*}
		%
		Since $|I_{\omega}^{v}|\leq \exp(-\frac{c_{\Phi}\theta(i+2,\omega,s_{i+2})}{2})$ by the definition of $c_{\Phi}$, we get 
		\begin{eqnarray*}
			\frac{|I_{\omega}^{v'}|}{|I_{\omega}^{v}|}&\leq& \exp\left(-\frac{c_{\Phi}\theta(i+2,\omega,s_{i+2})}{2}\cdot\frac{-4(C+1)(\varepsilon_{i+1})^4)}{c_{\Phi}}\right)\\
			&\leq&|I_{\omega}^{v}|^{\frac{-4(C+1)(\varepsilon_{i+1})^4)}{c_{\Phi}}}\leq |I_{\omega}^{v}|^{-(\varepsilon_{i+1})^3},
		\end{eqnarray*}
		so
		\begin{equation*}
		(2\ell_{\omega}^{w\ast v(k,l)})^{\xi}\leq 2^{\xi}|I_{\omega}^{v'}|
		\leq 2^{\xi}|I_{\omega}^{v}|^{1-(\varepsilon_{i+1})^3}\leq |I_{\omega}^{v}|^{1-(\varepsilon_{i+1})^2}.
		\end{equation*}
		So that \eqref{control overline{J_l}-{xi}} follows.
		\medskip

		For $k$ large enough so that $n^{i+1}_{k}\geq\frac{n}{(\varepsilon_{i+1})^2}$, define
		\begin{equation*}
		G^w(i+1,d_{i+1},k)=\{\underline{B_l},B_l\in D^w(i+1,d_{i+1},k)\}.
		\end{equation*}
		If $J_1$ and $J_2$ are two distinct elements of $G^w(i+1,d_{i+1},k)$ then their distance is at least $\max_{i\in \{1,2\}}(|\widehat{J_i}|/2-(|\widehat{J_i}|/2)^{\xi})$, which is larger than $\max_{i\in \{1,2\}}|\widehat{J_i}|/3$ for $k$ large enough (since $\xi>1$).
		
		We can define $m_{\xi}^{\{d_j\}_{1\leq j\leq i+1}}$ with $d_{i+1}\in\mathcal{D}_{i+1}$ as follows,
		\begin{equation*}\label{}
		m_{\xi}^{\{d_j\}_{1\leq j\leq i+1}}(J)=\frac{\zeta_{\omega,w,q_{i+1}}(\widehat{J})}{\sum_{J_l\in G^w(i+1,d_{i+1},k)}\zeta_{\omega,w,q_{i+1}}(\widehat{J_l})}\left(m_{\xi}^{\{d_j\}_{1\leq j\leq i}}(\overline{I_{\omega}^w})\right).
		\end{equation*}
		For any $J\in  G^w(i+1,d_{i+1},k)$, from the inequality \eqref{control  measure i-1} we get obtain
		\begin{eqnarray*}
			\zeta_{\omega,w,q_{i+1}}(\widehat{J}) & \leq & \left(\frac{|\widehat{J}|}{|I_{\omega}^w|}\right)^{\mathcal{T}^*(d_{i+1})-\frac{\varepsilon_{i+1}}{2}}\\
			& \leq & |J|^{\frac{(\mathcal{T}^*(d_{i+1})-\frac{\varepsilon_{i+1}}{2})(1-(\varepsilon_{i+1})^2)}{\xi}}|I_{\omega}^w|^{-\mathcal{T}^*(d_{i+1})}\\
			&\leq &|J|^{\frac{\mathcal{T}^*(d_{i+1})-\frac{2\varepsilon_{i+1}}{3}}{\xi}}|I_{\omega}^w|^{-\mathcal{T}^*(d_{i+1})}.
		\end{eqnarray*}
		Then, the inequality
		\begin{equation*}\label{}
		\sum_{J_l\in G^w(n,i+1,k)}\zeta_{\omega,w,q_{i+1}}(\widehat{J_l})\geq \frac{1}{4\Gamma_1},
		\end{equation*}
		yields, $\forall J\in G^w(i+1,d_{i+1},k)$:
		$$m_{\xi}^{\{d_j\}_{1\leq j\leq i+1}}(J)\leq 4\Gamma_1 |J|^{\frac{\mathcal{T}^*(d_{i+1})-\frac{2\varepsilon_{i+1}}{3}}{\xi}}|I_{\omega}^w|^{-\mathcal{T}^{*}(d_{i+1})}.$$

		Let $k_{i+1}>\kappa'_{i+1}$ large enough so that $n^{i+1}_{k_{i+1}}>\frac{n}{(\varepsilon_{i+1})^2}$. For any $d_{i+1}\in\mathcal D_{i+1}$ and
		for any $J\in G^w(i+1,d_{i+1}, k_{i+1})$, one has
		\begin{eqnarray*}
			&&4\Gamma_1|I_{\omega}^w|^{-\mathcal{T}^{*}(d_{i+1})} \\& \leq &4\Gamma_1(\exp(-Cn))^{-\mathcal{T}^{*}(d_{i+1})} \\
			&\leq&\exp(C\mathcal{T}^{*}(d_{i+1})n+\log(4\Gamma_1)) \\
			&\leq&\exp\left (\frac{c_{\Phi}}{6\xi}\cdot\frac{n}{\varepsilon_{i+1}}\right)\ (\text{noticing that } \mathcal{T}^{*}\leq -\mathcal{T}(0)=1, \log(4\Gamma_1)\leq n \text{ and } \varepsilon_{i+1}\leq
			\frac{c_{\Phi}}{6\xi(C+1)})\\
			&\leq&
			\exp\left (\frac{c_{\Phi}\varepsilon_{i+1}}{6\xi}(n+M_{i+1}+n^{i+1}_{k_{i+1}})\right)=\left (\exp\left (-\frac{c_{\Phi}}{2}(n+M_{i+1}+n^{i+1}_{k_{i+1}}\right )\right)^{-\frac{\varepsilon_{i+1}}{3\xi}}\\
			&\leq&|\widehat{J}|^{-\frac{\varepsilon_{i+1}}{3\xi}}\leq|J|^{-\frac{\varepsilon_{i+1}}{3\xi}},
		\end{eqnarray*}
		where the second inequality in the last line comes from \eqref{control I-p} and we had chosen $n$ large enough in the first step.
		Consequently, 
		\begin{equation*}
		\forall J\in G^w(i+1,d_{i+1}, k_{i+1}),\ m_{\xi}^{\{d_j\}_{1\leq j\leq i+1}}(J)\leq  |J|^{\frac{\mathcal{T}^*(d_{i+1})-\varepsilon_{i+1}}{\xi}}.
		\end{equation*}
		For $(d_j)_{1\leq j\leq i+1}\in\prod_{j=1}^{i+1}\mathcal{D}_j$ define:
		$$G(d_1,d_2,\cdots, d_{i+1})=\bigcup_{w\in G(d_1,d_2,\cdots, d_{i})}G^w(i+1,d_{i+1}, k_{i+1}),$$
		and
		$$G_{i+1}=\bigcup_{w\in G(i)}\bigcup_{q_{i+1}\in \mathcal{Q}_{i+1}}G^w(i+1,d_{i+1}, k_{i+1}).$$
		The definition of $m_{\xi}^{\{d_j\}_{1\leq j\leq i+1}}$ can be extended to the algebra generated by $$\bigcup_{s\leq i+1} G(d_1,d_2,\cdots, d_s),$$ and for any $J=I_{\omega}^v\in G(d_1,d_2,\cdots, d_s)$, $$m_{\xi}^{\{d_j\}_{1\leq j\leq i+1}}(J)\leq |J|^{\frac{\mathcal{T}^*(d_{i+1})-\varepsilon_{i+1}}{\xi}}.$$
		
		\item[Step 3] For any $\widetilde{d}=\{d_i\}_{i\in \mathbb{N}}\in\prod_{i=1}^\infty \mathcal D_i$, for any $J\in G(d_1,\cdots, d_i)$, define
		$m_{\xi}^{\widetilde{d}}(J)=m_{\xi}^{\{d_j\}_{1\leq j\leq i}}(J)$. This yields a probability measure $m_{\xi}^{\widetilde{d}}$  on the algebra generated by $\bigcup_{i\in \mathbb{N}} G(d_1,\cdots, d_i)$.
		
		For  any $i\in \N$, the elements in $G(d_1,\cdots, d_i)$ are closed and disjoint intervals. Also, for any $J\in G(d_1,\cdots, d_i)$, let $\widehat{J}$  be the ball associated with $J$. We have the following properties:
		
		\begin{enumerate}
			\item\label{i}  			
			\begin{itemize}
				\item $J\subset \widehat{J},$  for any $J\in G(d_1,\cdots, d_i)$;
				\item for any $J\in G(d_1,\cdots, d_i)$
				\begin{equation*}
				|J|\leq |\widehat{J}|^{\xi}\leq |J|^{1-(\varepsilon_{i})^3};
				\end{equation*}
				\item if $J_1\neq J_2$ belong to $G(d_1,\cdots, d_i)$, their distance is at least $\max_{l\in\{1,2\}}\frac{|\widehat{J_l}|}{3}$;
				\item The intervals $\widehat{J_l}$, $J_l\in G(d_1,\cdots, d_i)$, are disjoint.
			\end{itemize}
			\item\label{ii} For any $J$ in $G(d_1,d_2,\cdots,d_i)$, 
			$\widehat{J}\cap E(\omega,i,w,q_i)\neq \emptyset$, where $q_i\in \mathcal{Q}_i$ is such that $\mathcal{T}_{j_i}'(q_i)=d_i$ and $ E(\omega,i,w,q_i)$ is the set used in step 2.
			\item\label{iii} For any $J\in G(d_1,d_2,\cdots,d_i)$,
			\begin{equation}\label{F control}
			m_{\xi}^{\widetilde{d}}(J)\leq |J|^{\frac{\mathcal{T}^*(d_{i})-\varepsilon_{i}}{\xi}}.
			\end{equation}
			\item\label{iv} Any $J$ in $G(d_1,d_2,\cdots,d_i)$ is contained in  some element $L=\overline{I_{\omega}^w}\in G(d_1,d_2,\cdots,d_{i-1})$ such that
			\begin{equation*}\label{}
			m_{\xi}^{\widetilde{d}}(J)\leq 4\Gamma_1 m_{\xi}^{\widetilde{d}}(L)\zeta_{\omega,w,q_{i}}(\widehat{J}),
			\end{equation*}
			where $q_i\in \mathcal{Q}_i$ is such that $\mathcal{T}_{j_i}'(q_i)=d_i$.
		\end{enumerate}
		
		Because of the separation property \ref{i}, we get a probability measure $m_{\xi}^{\widetilde{d}}$ on $\sigma(J:J\in\bigcup_{i\geq 1} G(d_1,d_2,\cdots,d_i))$ such that properties \ref{i} to \ref{iv} hold for every $i\geq 1$. We now define
		$$K(\xi,\widetilde{d})=\bigcap_{i\geq 1}\bigcup_{J\in G(d_1,\cdots,d_i)}J,$$
		then, $m_{\xi}^{\widetilde{d}}(K(\xi,\widetilde{d}))=1$. The measure $m_{\xi}^{\widetilde{d}}$ can be extended to $[0,1]$ by setting, for any $B\in \mathcal{B}([0,1])$, $m_{\xi}^{\widetilde{d}}(B):=m_{\xi}^{\widetilde{d}}(B\cap K(\xi,\widetilde{d}))$.
		
		\item[Step 4] Fix a sequence $\widetilde{d}=\{d_i\}_{i\in \mathbb{N}}\in \prod_{i\in\mathbb{N}}\mathcal{D}_i$ such that $$\lim_{i\to \infty} d_i=d,\ \lim_{i\to \infty } \mathcal{T}^*(d_i)=\mathcal{T}^*(d).$$
		Define $K^d(\xi)=K(\xi,\widetilde{d})$, and $m^d_\xi=m_{\xi}^{\widetilde{d}}$.
		
		From the construction, we claim that: $K^d(\xi)\subset S(\omega,d,\xi,\widetilde{\epsilon})$.
		In fact we just need to prove the following lemma: 
		\begin{lemma}\label{lemma converge to d}
			For any $w\in G(d_1,d_2,\cdots, d_i)$ with $n=|w|$, for any $v\in G^w(i+1,d_{i+1},k)$ with $n^{i+1}_{k}\geq\frac{n}{(\varepsilon_{j+1})^2}$, for any $\vl\in [w\ast v]_{\omega}$, we have:
			$$\left|\frac{S_{n+M_{i+1}+n^{i+1}_{k}}\Psi(\omega,\vl)}{S_{n+M_{i+1}+n^{i+1}_{k}}\Phi(\omega,\vl)}-d\right|\leq \epsilon'_{i+1},$$
			where $\epsilon'_{i+1}=|d-d_{i+1}|+2\varepsilon_{i+1}$.
		\end{lemma}
		\begin{proof}
			First, for any $\Upsilon\in\{\Phi,\Psi\}$,
			\begin{eqnarray*}
				&&\left|S_{n+M_{i+1}+n^{i+1}_{k}}\Upsilon(\omega,\vl)-S_{n^{i+1}_{k}}\Upsilon_{j_{i+1}}(F^{n+M_{i+1}}(\omega,\vl))\right|\\
				&\leq&2(n+M_{i+1})C+|S_{n^{i+1}_{k}}(\Upsilon-\Upsilon_{j_{i+1}})(F^{n+M_{i+1}}(\omega,\vl))|\\
				&\leq& 2(n+M_{i+1})C+S_{n^{i+1}_{k}}\mathrm{var}_{j_{i+1}}\Upsilon(\sigma^{n+M_{i+1}}\omega)\\
				&\leq&4n^{i+1}_{k}(\varepsilon_{j+1})^2+2n^{i+1}_{k}(\varepsilon_{j+1})^4\ (\text{ see } \eqref{control Sp-var-ji} \text{ and } \int_{\Omega} \var_{j_{i+1}}\Upsilon\ d\mathbb{P}\le (\varepsilon_{i+1})^4)\\        	 
				&\leq&5n^{i+1}_{k}(\varepsilon_{j+1})^2.
			\end{eqnarray*}
			Next, applying again the shorthand $\widetilde F^n=F^{n+M_{i+1}}$, 
			\begin{eqnarray*}
				&&\left|\frac{S_{n+M_{i+1}+n^{i+1}_{k}}\Psi(\omega,\vl)}{S_{n+M_{i+1}+n^{i+1}_{k}}\Phi(\omega,\vl)}-d\right|\\
				&\leq&\left|\frac{S_{n+M_{i+1}+n^{i+1}_{k}}\Psi(\omega,\vl)}{S_{n+M_{i+1}+n^{i+1}_{k}}\Phi(\omega,\vl)}-d_{i+1}\right|+|d-d_{i+1}|\\
				&\leq&\left|\frac{S_{n+M_{i+1}+n^{i+1}_{k}}\Psi(\omega,\vl)-S_{n^{i+1}_{k}}\Psi_{j_{i+1}}(\widetilde F^n(\omega,\vl))}{S_{n+M_{i+1}+n^{i+1}_{k}}\Phi(\omega,\vl)}\right|\\
				&&+\left|\frac{S_{n^{i+1}_{k}}\Psi_{j_{i+1}}(\widetilde F^n(\omega,\vl))-d_{i+1}S_{n^{i+1}_{k}}\Phi_{j_{i+1}}(\widetilde F^n(\omega,\vl))}{S_{n+M_{i+1}+n^{i+1}_{k}}\Phi(\omega,\vl)}\right|\\
				&&+d_{i+1}\left|\frac{S_{n^{i+1}_{k}}\Phi_{j_{i+1}}(\widetilde F^n(\omega,\vl))-S_{n+M_{i+1}+n^{i+1}_{k}}\Phi(\omega,\vl)}{S_{n+M_{i+1}+n^{i+1}_{k}}\Phi(\omega,\vl)}\right|\\
				&&+|d-d_{i+1}|\\
				&\leq&\left|\frac{(\varepsilon_{j+1})^2 S_{n^{i+1}_{k}}\Phi_{j_{i+1}}(\widetilde F^n(\omega,\vl))}{S_{n+M_{i+1}+n^{i+1}_{k}}\Phi(\omega,\vl)}\right|+\left|\frac{5(d_{i+1}+1)n^{i+1}_{k}(\varepsilon_{j+1})^2}{S_{n+M_{i+1}+n^{i+1}_{k}}\Phi(\omega,\vl)}\right|+|d-d_{i+1}|,
			\end{eqnarray*}	
			where we have used \eqref{control psi/Phi} and \eqref{control Sp-var-ji}. Thus, 
			\begin{eqnarray*}     
				&&\left|\frac{S_{n+M_{i+1}+n^{i+1}_{k}}\Psi(\omega,\vl)}{S_{n+M_{i+1}+n^{i+1}_{k}}\Phi(\omega,\vl)}-d\right|\\  &\leq&(\varepsilon_{j+1})^2+\left|\frac{5n^{i+1}_{k}(\varepsilon_{j+1})^4+5(d_{i+1}+1)n^{i+1}_{k}(\varepsilon_{j+1})^2}{\frac{(n+M_{i+1}+n^{i+1}_{k})c_{\Phi}}{2}}\right|+|d-d_{i+1}|\\
				&\leq&(\varepsilon_{j+1})^2+\frac{10(\varepsilon_{j+1})^4+10(d_{i+1}+1)(\varepsilon_{j+1})^2}{c_{\Phi}}+|d-d_{i+1}|\\
				&\leq&(\varepsilon_{j+1})^2+\varepsilon_{j+1}+|d-d_{i+1}|\leq\epsilon'_{i+1}.
			\end{eqnarray*}	
		\end{proof}

		Now we turn to  estimate the lower Hausdorff dimension of $m^d_\xi$. If $\mathcal T^*(d)=0$, there is nothing need to prove. So we assume that $\mathcal T^*(d)>0$.
		
		For any $J\in G(d_1,d_2,\cdots,d_i)$, define $g(J)=i$.	
		Let us fix $B$ a subinterval  of $[0,1]$ of length smaller than that of every element in $G(d_1)$, and assume that $B\cap K_{\xi}^{\widetilde{d}}\neq \emptyset$. Let $L=\overline{I_{\omega}^w}$ be the element of largest diameter in $\bigcup_{i\geq 1} G(d_1\cdots d_i)$ such that $B$ intersects at least two elements of $G(d_1\cdots d_{g(L)+1})$ and is included in $L\in G_{d_1\cdots d_{g(L)}}$. We remark that this implies that $B$ does not intersect any other element of $G(d_1,d_2,\cdots,d_i)$, where $i=g(L)$, and as a consequence $m^d_\xi(B)\leq m^d_\xi(L)$.
		
		\medskip
		
		Let us distinguish three cases:

		$\bullet$ $|B|\geq |L|$: then
		\begin{equation}\label{control B>L}
		m^d_\xi(B)\leq m^d_\xi(L)\leq |L|^{\frac{(\mathcal{T}^{*}(d_{i})-\varepsilon_{s})}{\xi}}\leq |B|^{\frac{(\mathcal{T}^{*}(d_{i})-\varepsilon_{i})}{\xi}}.
		\end{equation}
		
		$\bullet$  $|B|\leq \frac{1}{4}|L|\exp(-|w|(\varepsilon_{i+1})^2)$.
		Assume $L_1,\dots,L_p$ are the elements of $G_{i+1}$ which have non-empty intersection with $B$.
		From property \ref{iv}, we can choose $q_{i+1}\in\mathcal{Q}_{i+1}$ so that $\mathcal{T}_{j_{i+1}}'(q_{i+1})=d_{i+1}$, and  get
		\begin{equation*}
		m^d_\xi(B)=\sum_{l=1}^p m^d_\xi(B\cap L_l)\leq 4\Gamma_1 m^d_\xi(L)\sum_{i=1}^p \zeta_{\omega,w, q_{i+1}}(\widehat{L_l}).
		\end{equation*}
		
		From property \ref{i} we can also deduce that $\max\{|\widehat{L_l}|:1\leq l\leq p\}\leq 3|B|$. From property \ref{ii} we can get $E(\omega,i+1,w,q_{i+1})\cap\widehat{L_l}\neq\emptyset$. If $y$ is taken in the intersection, we have $B(y,4|B|)\supset(\bigcup_{l=1}^p\widehat{L_l}) $.
		
		Now we notice that $L$ is the closure of $I_{\omega}^{w}$ for some $w\in\Sigma_{\omega,n}$ with $n\in\N$, and we have  $\sigma^n\omega\in {\Omega}_{i+1}$. Using \eqref{control zeta B(y,r)} in lemma \ref{lemma control zeta B(y,r)}, we can get:  
		$$\zeta_{\omega,w,q_{i+1}}(B(y,4|B|))\leq \left(\frac{4|B|}{|I_{\omega}^{w}|}\right )^{\mathcal{T}^*(d_{i+1})-\frac{\varepsilon_{i+1}}{2}}.$$

		Now, since $L=\overline{I_{\omega}^{w}}$ is the closure of $I_{\omega}^{w}$, we have:
		\begin{equation*}\label{}
		\begin{split}
		m^d_\xi(B) \leq& 4\Gamma_1 m^d_\xi(L)\sum_{l=1}^p \zeta_{\omega,v,q_{i+1}}(\widehat{L_l})  \\
		\leq&  4\Gamma_1 m^d_\xi(L)\zeta_{\omega,w,q_{i+1}}(B(y,4|B|))\\
		\leq&  4\Gamma_1|L|^{\frac{\mathcal{T}^{*}(d_{i})-\varepsilon_{i}}{\xi}}\left(\frac{4|B|}{|L|}\right)^{\mathcal{T}^*(d_{i+1})-\frac{\varepsilon_{i+1}}{2}}
		\leq 4\Gamma_1(4|B|)^{\frac{\mathcal{T}^{*}(d_{i})-\varepsilon_{i}}{\xi}} \left(\frac{4|B|}{|L|}\right)^{\alpha_i},
		\end{split}
		\end{equation*}
		where $\alpha_i=\mathcal{T}^*(d_{i+1})-\frac{\varepsilon_{i+1}}{2}-\frac{\mathcal{T}^{*}(d_{i})-\varepsilon_{i}}{\xi}$ is positive for $i$ large enough since $\lim_{i\to\infty}\mathcal{T}^{*}(d_{i})=\mathcal T^*(d)>0$. Moreover, $4|B|/|L|\le 1$, so
		\begin{equation}\label{control B<<L}
		m^d_\xi(B) \leq 4\Gamma_1(4|B|)^{\frac{\mathcal{T}^{*}(d_{i})-\varepsilon_{i}}{\xi}}.
		\end{equation}
		
		$\bullet$  $\frac{1}{4}|L|\exp(-|w|(\varepsilon_{i+1})^2)\leq |B|\leq |L|$:
		
		We need at most $M(B)=\lfloor4\exp(|w|(\varepsilon_{i+1})^2)\rfloor+1$ contiguous intervals $(B(k))_{1\leq k \leq M(B)}$ with diameter $\frac{1}{4}|L|\exp(-|w|(\varepsilon_{i+1})^2)$ to cover $B$. For these intervals we have the estimate above. Consequently,
		
		\begin{equation*}\label{}
		\begin{split}
		m^d_\xi(B) \leq& \sum_{k=1}^{M(B)}4\Gamma_1(4|B(k)|)^{\frac{\mathcal{T}^{*}(d_{i})-\varepsilon_{i}}{\xi}}\\
		\leq&  4 \Gamma_1 M(B) (4|B|)^{\frac{\mathcal{T}^{*}(d_{i})-\varepsilon_{i}}{\xi}} \\
		\leq& 20\Gamma_1\exp(|w|(\varepsilon_{i+1})^2)(4|B|)^{\frac{\mathcal{T}^{*}(d_{i})-\varepsilon_{i}}{\xi}} .\\
		\end{split}
		\end{equation*}
		Since $|B|\leq |L|\leq \exp(-\frac{n c_{\Phi}}{2})$, we get $\exp(|w|(\varepsilon_{i+1})^2)\leq |L|^{-\frac{2(\varepsilon_{i+1})^2}{c_{\Phi}}}\leq |B|^{-\frac{2(\varepsilon_{i+1})^2}{c_{\Phi}}}.$
		Finally,			\begin{equation}\label{control B<L}
		m^d_\xi(B)\leq 20\Gamma_1(4|B|)^{\frac{\mathcal{T}^{*}(d_{i})-\varepsilon_{i}}{\xi}}|B|^{-\frac{2(\varepsilon_{i+1})^2}{c_{\Phi}}}.
		\end{equation} 
		
		It follows from the estimations \eqref{control B>L},\eqref{control B<<L} and \eqref{control B<L} that 			$\underline \dim_H (m^d_\xi)\geq\frac{\mathcal{T}^{*}(d)}{\xi}$.
		
	\end{description}
	
	We have finished the proof of theorem \ref{ubiqu}.
\end{proof}

\section{Conclusion on the lower bound for the lower Hausdorff spectrum}
Next proposition is both a complement to proposition~\ref{lower bound-1}, and an improvement over the interval $[\mathcal T'(+\infty),\mathcal{T}'(t_0-))$.
\begin{proposition}\label{lower bound-2}
	For $\mathbb P$-a.e. $\omega$, for any $d\in [0,\mathcal{T}'(t_0-)]$, one has
	$$\dim_H(\El(\nu_{\omega},d))\geq t_0 d=(\dim_H X_{\omega})\cdot d .$$
\end{proposition}
\begin{proof}
	%
	If $d\in (0,\mathcal{T}'(t_0-)]$, we write $d=\mathcal{T}'(t_0-)/\xi$ with $\xi\geq 1$. We can find a suitable sequence $\widetilde{\varepsilon} $ such that theorem \ref{ubiqu} and remark~\ref{S to El} hold. This provides us with a positive Borel measure $m^{\mathcal{T}'(t_0-)}_{\xi}$ on $K^{\mathcal{T}'(t_0-)}(\xi)$, with the following properties:
	\begin{itemize}
		\item $m^{\mathcal{T}'(t_0-)}_{\xi}(K^{\mathcal{T}'(t_0-)}(\xi))=1$ and $\dim_H (m^{\mathcal{T}'(t_0-)}_{\xi})\geq \frac{\mathcal{T}^*(\mathcal{T}'(t_0-)) }{\xi}=d\, t_0$.
		\item $m^{\mathcal{T}'(t_0-)}_{\xi}(E)=0$ as soon as $\dim_H E< d\, t_0$.
		\item For any $x\in K^{\mathcal{T}'(t_0-)}(\xi)$, we have that $\dl(\vw,x)\leq d$.
	\end{itemize}
	It follows from lemma \ref{prop8.3} that
	$$(K^{\mathcal{T}'(t_0-)}(\xi)\setminus (\bigcup_{0\leq h<d}F(h)))\subset (\El(\vw,d)\cup \Xi_\omega). $$
	Also, corollary \ref{control Fh} tells $\dim_H F(h)\leq ht_0  < d t_0$ for all $0\leq h< d$, so $m^{\mathcal{T}'(t_0-)}_{\xi}(F(h))=0$ for all $0\leq h< d$. Moreover, the family of sets $(F(h))_{0< h< d}$ is nondecreasing. Thus, we have
	$$m^{\mathcal{T}'(t_0-)}_{\xi}(\El(\vw,d)\cup \Xi_\omega)>0,$$
	hence
	$$\dim_H (\El(\vw,d)\cup \Xi_\omega)\geq dt_0.$$
	Finally, $\dim_H \El(\vw,d)\geq d t_0$ since $\Xi_\omega$ is a countable set.
	
	If $d=0$ or $t_0=0$, we have $$\emptyset\neq\Xi'_\omega\subset \El(\vw,0),$$
	thus $\dim_H \El(\vw,d)\geq d t_0$ for $d=0$.
\end{proof}

Next proposition collects all the information required to conclude regarding the lower bound for the lower Hausdorff spectrum. Its claim (iii) is the desired sharp lower bound.

\begin{proposition}\label{lower bound-3}For $\mathbb P$-a.e. $\omega$:
	\begin{enumerate}
		\item if $d\in [0,\mathcal{T}'(t_0-)]$, then $\dim_H(\El(\nu_{\omega},d))\geq t_0 d$,
		\item if $d\in [\mathcal{T}'(+\infty),\mathcal{T}'(-\infty)]$, then $\dim_H(\El(\nu_{\omega},d))\geq \mathcal{T}^*(d)$,
		\item for any $d\in [0,\mathcal{T}'(-\infty)]$, $\dim_H(\El(\nu_{\omega},d))\geq \widetilde{\mathcal{T}}^{*}(d)$.
	\end{enumerate}
\end{proposition}
\begin{proof}
	(i) and (ii) come from proposition~\ref{lower bound-2} and proposition~\ref{lower bound-1}.
	
	To prove (iii), since $\widetilde{\mathcal{T}}(q)=\min\{\mathcal{T}(q),0\}$,$\mathcal{T}(t_0)=0$ and $\mathcal{T}$ is increasing,
	$$\widetilde{\mathcal{T}}^{*}(d)=\inf_{q\in\mathbb{R}}\{qd-\widetilde{\mathcal{T}}(q)\}=\left\{
	\begin{array}{ll}
	t_0 d, & d\in [0,\mathcal{T}'(t_0-)], \\
	\mathcal{T}^*(d), & d\in [\mathcal{T}'(t_0-),\mathcal{T}'(-\infty)].
	\end{array}
	\right.
	$$		
\end{proof}

\section{Hausdorff dimensions of the level sets $E(\nu_{\omega},d)$ and $\overline E(\nu_{\omega},d)$}

Recall that $v(\omega,n,x)$ has been defined in definition~\ref{def 4.1}. We need to introduce another approximation rate.

For any $x\in [0,1]\setminus\{x_{\omega}^{vs}:v\in\Sigma_{\omega,\ast},s\in S'(\omega,v,1)\} $, define
$$\widehat\xi(\omega,n,x)=\frac{\log(\inf\{|x-x_{\omega}^{vs}|:|v|\leq n, s\in S'(\omega,v,1)\})}{\log |I_{\omega}^{v(\omega,n,x)}|}$$
and then
$$\widehat\xi(\omega,x)=\liminf_{n\to \infty}\widehat\xi(\omega,n,x).$$
The desired conclusion on the sets $E(\nu_{\omega},d)$ and $\overline E(\nu_{\omega},d)$ will follow from  two lemmas and one proposition.
\begin{lemma}\label{lem AR=1}
	For $\mathbb P$-a.e. $\omega\in \Omega$, we have
	$$\{x\in [0,1]\setminus\Xi_\omega: \widehat\xi(\omega,x)>1\}=\emptyset.$$
	In other words, for any $x\in [0,1]$, if $x\notin\Xi_\omega$, then  $\widehat\xi(\omega,x)=1.$
\end{lemma}
\begin{proof}
	We just need to prove that for any $k\in \mathbb{Z}^+$,
	$$\{x\in[0,1]\setminus\Xi_\omega:\widehat\xi(\omega,x)>1+1/k\}=\emptyset.$$
	
	For any $x\in [0,1]\setminus\Xi_\omega$ such that $\widehat\xi(\omega,x)>1+1/k$, there exists $N(x)\in \mathbb{Z}^+$ such that for any $n\geq N(x)$ one has
	$$\inf\{|x-x_{\omega}^{vs}|:|v|\leq n, s\in S'(\omega,v,1)\}\leq |I_{\omega}^{v(\omega,n,x)}|^{1+1/k}.$$	
	Furthermore, the infimum must be  attained at a point $x_{\omega}^{vs}$ which is in the closure of $I_{\omega}^{v(\omega,n,x)}$. We denote  $vs$ by $w(\omega,n,x)$. We just need to prove that  $x$ is the point $x_{\omega}^{w(\omega,n,x)}$ for $n$ large enough. And $x_{\omega}^{w(\omega,n,x)}\in \Xi_\omega$ yields the lemma.
	
	The choice of $x^{w(\omega,n+1,x)}$ must be made in  $\{x^{v(\omega,n+1,x)s}:\ s\in S'(\omega,v(\omega,n+1,x),1)\}\cup\{x^{w(\omega,n,x)}\}$. Otherwise it is easily seen  that it is in  contradiction with  the choice  of $w(\omega,n,x)$ and $x\in I_{\omega}^{v(\omega,n+1,x)}$.
	
	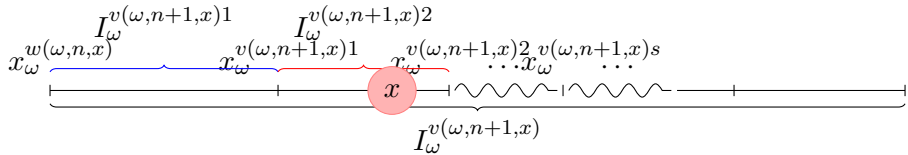
\begin{figure}[!htb]
		\begin{tikzpicture}[scale=0.75]
		\draw[-] (0,0) -- (7,0);
		\draw[snake=coil,segment aspect=0] (7.1,0) -- (8.9,0);
		\node (b0) at (8,0.4)   {$\cdots$};
		\draw[snake=coil,segment aspect=0] (9.1,0) -- (10.9,0);
		\node (b0) at (10,0.4)   {$\cdots$};
		\draw[-] (11,0) -- (15,0);
		\draw[snake=brace,mirror snake,raise snake=5pt,]   (0,0) -- (15,0);
		\node (b1) at (7.5,-0.8)   {$I_{\omega}^{v(\omega,n+1,x)}$};
		\draw[-] (0,0.1) -- (0,-0.1);
		\node (b2) at (0.2,0.6)   {$x_{\omega}^{w(\omega,n,x)}$};
		\draw[-] (4,0.1) -- (4,-0.1);
		\node (b3) at (4.2,0.6)   {$x_{\omega}^{v(\omega,n+1,x)1}$};
		\draw[-] (7,0.1) -- (7,-0.1);
		\node (b4) at (7.2,0.6)   {$x_{\omega}^{v(\omega,n+1,x)2}$};
		\draw[-] (9,0.1) -- (9,-0.1);
		\node (b5) at (9.5,0.6)   {$x_{\omega}^{v(\omega,n+1,x)s}$};
		\draw[-] (12,0.1) -- (12,-0.1);
		\draw[-] (15,0.1) -- (15,-0.1);
		\draw[snake=brace,snake,raise snake=7pt,blue]   (0,0) -- (4,0);
		\node (b6) at (2,1.2)   {$I_{\omega}^{v(\omega,n+1,x)1}$};
		\draw[snake=brace,snake,raise snake=7pt,red]   (4,0) -- (7,0);
		\node (b7) at (5.5,1.2)   {$I_{\omega}^{v(\omega,n+1,x)2}$};
		\node at (6,0) [circle,draw=red!50,fill= red!30] {$x$};
		\end{tikzpicture}
		\caption{The choice for $w(\omega,n+1,x)$}
		\label{fig9}
	\end{figure}
	
	We have
	\begin{eqnarray*}
		&&\inf\{|x-x_{\omega}^{vs}|:|v|\leq n+1, s\in S'(\omega,v,1)\}\\
		&=&|x_{\omega}^{w(\omega,n+1,x)}-x|\leq |I_{\omega}^{v(\omega,n+1,x)}|^{1+1/k}
		\leq |I_{\omega}^{v(\omega,n,x)}|^{1+1/k} .
	\end{eqnarray*}
	Now suppose that $x^{w(\omega,n+1,x)}\neq x^{w(\omega,n,x)}$. On the one hand, since $x^{w(\omega,n,x)}$ must be an endpoint of $I_{\omega}^{v(\omega,n+1,x)}$, there exists $s\in S(\omega,v(\omega,n+1,x),1)$ such that
	$$|I_{\omega}^{v(\omega,n,x)s}|\le |x_{\omega}^{w(\omega,n+1,x)}-x_{\omega}^{w(\omega,n,x)}|\leq |x_{\omega}^{w(\omega,n+1,x)}-x|+|x_{\omega}^{w(\omega,n,x)}|\leq 2|I_{\omega}^{v(\omega,n,x)}|^{1+1/k},$$
	and on the other hand,  from ~\eqref{cphi} and proposition~\ref{for n}, it is easy to prove that for $n$ large enough we have
	$$|I_{\omega}^{v(\omega,n,x)s}|=|I_{\omega}^{v(\omega,n,x)}|\cdot e^{o(\log(|I_{\omega}^{v(\omega,n,x)}|))}\geq |I_{\omega}^{v(\omega,n,x)}|^{1+\frac{1}{2k}}> 4|I_{\omega}^{v(\omega,n,x)}|^{1+1/k},$$
	which yields a the contradiction. Thus, $x^{w(\omega,n+1,x)}=x^{w(\omega,n,x)}$ for $n$ large enough, i.e. it is the point $x$.
\end{proof}

\begin{lemma}\label{lemAR2} For $\mathbb P$-a.e. $\omega\in \Omega$, for any $x\in[0,1]\setminus\Xi_\omega$, we have
	$\du(\nu_{\omega},x)\geq \liminf_{n\to \infty}\frac{\log \nu_{\omega}(\widering{I}_{\omega}^{v(\omega,n,x)})}{\log |I_{\omega}^{v(\omega,n,x)}|},$ where $v(\omega,n,x)$ is defined after definition~\ref{def 4.1} in section~\ref{section:Pointwise behavior} .
\end{lemma}

\begin{proof} Suppose that the conclusion of lemma \ref{lem AR=1} holds for $\omega$. If $x\in[0,1]\setminus\Xi_\omega$,  there exists a subsequence $\{n_k\}_{k\in \mathbb{Z}^+}$ such that $\widehat\xi(\omega,n_k,x)\to 1$ as $k\to \infty$.
	Now,
	\begin{eqnarray*}
		\limsup_{r\to 0}\dfrac{\log \nu_{\omega}(B(x,r))}{\log r}
		&\geq&\limsup_{k\to \infty}\dfrac{\log \nu_{\omega}(B(x,|I_{\omega}^{v(\omega,n_k,x)}|^{\widehat\xi(\omega,n_k,x)+1/n_k}))}{\log |I_{\omega}^{v(\omega,n_k,x)}|^{\widehat\xi(\omega,n_k,x)+1/n_k}}\\
		&\geq& \limsup_{k\to \infty}\dfrac{\log \nu_{\omega}(\widering{I}_{\omega}^{v(\omega,n_k,x)})}{\log |I_{\omega}^{v(\omega,n_k,x)}|}\geq \liminf_{n\to \infty}\frac{\log \nu_{\omega}(\widering{I}_{\omega}^{v(\omega,n,x)})}{\log |I_{\omega}^{v(\omega,n,x)}|}.
	\end{eqnarray*}
	The second inequality follows from the fact that $\widehat\xi(\omega,n_k,x) \to 1$ as $k\to \infty$ and  
	
	$$B(x,|I_{\omega}^{v(\omega,n_k,x)}|^{\widehat\xi(\omega,n_k,x)+1/n_k})\subset \widering{I}_{\omega}^{v(\omega,n_k,x)}$$ by definition of $\widehat\xi(\omega,n_k,x)$. 	
\end{proof}

\begin{proposition}
	For $\mathbb P$-a.e. $\omega\in\Omega$,
	\begin{enumerate}
		\item\   if $d\in [\mathcal{T}'(+\infty),\mathcal{T}'(t_0-)]$, then $\dim_H(\{x\in [0,1]: \du(\nu_{\omega},x)\leq d\})\leq \mathcal{T}^{*}(d);$
		\item \ if $x\in [0,1]\setminus \Xi'_\omega$ then $\du(\nu_{\omega},x)\ge  \mathcal{T}'(+\infty)$.
		
		\item\
		$\overline{E}(\nu_{\omega},0)=E(\nu_{\omega},0)
		=\Xi'_\omega$,	so $\dim_H(\overline{E}(\nu_{\omega},0))=\dim_H(E(\nu_{\omega},0))=0.$
		
	\end{enumerate}
\end{proposition}
\begin{proof}$(i)$ For any $d\in [\mathcal{T}'(+\infty),\mathcal{T}'(t_0-)]$, for any $\varepsilon>0$, there exists $q\geq 0$ such that $\mathcal{T}^*(d)\geq qd-\mathcal{T}(q)-\varepsilon/2$. Choose $\epsilon>0$ such that $q\epsilon\leq \varepsilon/4$. Due to the previous lemma, for any $N\in \N$ we have
	\begin{eqnarray*}
		\{x\in [0,1]\setminus \Xi_\omega: \du(\nu_{\omega},x)\leq d\}&\subset&\left\{x\in [0,1]:\liminf_{n\to \infty}\frac{\log \nu_{\omega}(\widering{I}_{\omega}^{v(\omega,n,x)})}{\log |I_{\omega}^{v(\omega,n,x)}|}\leq d\right\}\\
		&\subset&\bigcup_{n\geq N}\bigcup_{v\in \Sigma_{\omega,n},\nu_{\omega}(\widering{I}_{\omega}^{v})\geq |I_{\omega}^{v}|^{d+\epsilon}}I_{\omega}^{v}.
	\end{eqnarray*}
	For any $\delta>0$, for $N$ large enough, and $v\in\Sigma_{\omega,n}$, one has $|I_{\omega}^{v}|<\delta$. Choosing $h=\mathcal{T}^*(d)+\varepsilon$ we get for $N$ large enough,
	\begin{multline*}
	\mathcal{H}_{\delta}^h(\{x\in [0,1]\setminus \Xi_\omega: \du(\nu_{\omega},x)\leq d\})\leq \sum_{n\geq N}\sum_{v\in \Sigma_{\omega,n}:\, \nu_\omega(\widering I_{\omega}^{v})\geq |I_{\omega}^{v}|^{d+\epsilon}}|I_{\omega}^{v}|^{\mathcal{T}^*(d)+\varepsilon}\\
	\leq \sum_{n\geq N}\sum_{v\in \Sigma_{\omega,n}:\, \nu_\omega(\widering I_{\omega}^{v})\geq |I_{\omega}^{v}|^{d+\epsilon}}|I_{\omega}^{v}|^{qd-\mathcal{T}(q)+\varepsilon/2}
	\leq\sum_{n\geq N}\sum_{v\in \Sigma_{\omega,n}}|I_{\omega}^{v}|^{-\mathcal{T}(q)+\varepsilon/2-q \epsilon }(\nu_\omega(\widering{I}_{\omega}^{v}))^{q}\\
	\leq \sum_{n\geq N}\sum_{v\in \Sigma_{\omega,n}}|I_{\omega}^{v}|^{-\mathcal{T}(q)+\varepsilon/4}(\nu_\omega(\widering{I}_{\omega}^{v}))^{q}
	\leq \sum_{n\geq N} \exp(n P(q\Psi-\mathcal{T}(q)\Phi)-\frac{nc_{\Phi}\varepsilon}{8})\\
	\leq \sum_{n\geq N} \exp(-\frac{nc_{\Phi}\varepsilon}{8}).
	\end{multline*}
	%
	Here we used the fact that $\nu_\omega(\widering{I}_{\omega}^{v})\leq |X_{\omega}^v|\leq |U_{\omega}^v|\leq\exp(S_n\Psi(\omega,\vl)+o(n))$ for any $\vl\in [v]_{\omega}$ and $v\in \Sigma_{\omega,*}$.
	
	Letting $N$ go to $\infty$   we get $\mathcal{H}^h_\delta(\{x\in [0,1]\setminus \Xi_\omega: \du(\nu_{\omega},x)\leq d\})=0$ for any $\delta>0$, so $\mathcal{H}^h(\{x\in [0,1]\setminus \Xi_\omega: \du(\nu_{\omega},x)\leq d\})=0$. This holds for any $h>-\mathcal{T}^*(d)$, so $\dim_H \{x\in [0,1]\setminus \Xi_\omega: \du(\nu_{\omega},x)\leq d\}\leq \mathcal{T}^*(d).$ Since $\Xi_\omega$ is countable we get the desired conclusion.
	
	\medskip
	
	$(ii)$
	%
	For any $v\in \Sigma_{\omega,\ast}$ define 
	$${I'}_{\omega}^v:=(F_{\mu_{\omega}}(m_{\omega}^v),F_{\mu_{\omega}}(M_{\omega}^v)]
	=F_{\mu_{\omega}}(X_{\omega}^v)\setminus\{F_{\mu_{\omega}}(m_{\omega}^v)\}.$$
	Then for $x\in (0,1]$ and $n\geq 1$, let ${v'}(\omega,n,x)$ be the unique element $v$ in $\Sigma_{\omega,n}$ such that $x\in {I'}_{\omega}^{v}$. If $x=0$,  $v'(\omega,n,0)$ is the unique $v\in\Sigma_{\omega,n}$ such that $1\in \overline{{I'}_{\omega}^{v}}$.
	In $(i)$ we proved that $\{x\in [0,1]:\liminf_{n\to \infty}\frac{\log \nu_{\omega}(\widering{I}_{\omega}^{v(\omega,n,x)})}{\log |I_{\omega}^{v(\omega,n,x)}|}\leq d\}=\emptyset$ for $d<\mathcal{T}'(+\infty)$.
	Using the same arguments we can also prove that 
	$\{x\in [0,1]:\liminf_{n\to \infty}\frac{\log \nu_{\omega}(\widering{{I'}}_{\omega}^{{v'}(\omega,n,x)})}{\log |{I'}_{\omega}^{{v'}(\omega,n,x)}|}\leq d\}=\emptyset$ for $d<\mathcal{T}'(+\infty)$.
	Thus,  for any $x\in [0,1]$,
	\begin{equation}\label{37}
	\min\left ( \liminf_{n\to \infty}\frac{\log \nu_{\omega}(\widering{I}_{\omega}^{v(\omega,n,x)})}{\log |I_{\omega}^{v(\omega,n,x)}|},  \liminf_{n\to \infty}\frac{\log \nu_{\omega}(\widering{{I'}}_{\omega}^{{v'}(\omega,n,x)})}{\log |{I'}_{\omega}^{{v'}(\omega,n,x)}|}\right )\geq \mathcal{T}'(+\infty).
	\end{equation}
	Now we come to the proof of the assertion. From the proof of item 1, we just need to deal with the set $\Xi_\omega\setminus \Xi'_\omega$. For $x\in \Xi_\omega\setminus \Xi'_\omega$ and $r>0$ small enough, there exist $n,n'\in \mathbb{N}$ such that $|I_{\omega}^{v(\omega,n,x)}|>r,\ |{I'}_{\omega}^{{v'}(\omega,n',x)}|>r$ and $|I_{\omega}^{v(\omega,n+1,x)}|\leq r,\ |{I'}_{\omega}^{{v'}(\omega,n'+1,x)}|\leq r$.
	Now we have $B(x,r)\subset \widering{I}_{\omega}^{v(\omega,n,x)}\cup \widering{I'}_{\omega}^{v'(\omega,n',x)}\cup\{x\}$.
	Since $\nu_{\omega}(\{x\})=0$, we have 
	$$\nu_{\omega}(B(x,r))\leq \nu_{\omega}(\widering{I}_{\omega}^{v(\omega,n,x)})+\nu_{\omega}({I'}_{\omega}^{{v'}(\omega,n',x)}),$$
	and it follows from \eqref{37} and the choices of $n$ and  $n'$ that  $\du(\nu_{\omega},x)\ge  \mathcal{T}'(+\infty)$.

	%
	%
	%
	%
	%
	\medskip
	$(iii)$ It clearly follows from $(ii)$.\end{proof}

\begin{theorem}
	For $\mathbb P$-a.e. $\omega\in \Omega$, for any $d\in [\mathcal{T}'(+\infty),\mathcal{T}'(-\infty)]$, we have
	$$\dim_H E(\nu_{\omega},d)=\dim_H \overline{E}(\nu_{\omega},d)=\mathcal{T}^*(d).$$
\end{theorem}
\begin{proof}
	The expected lower bound for the Hausdorff dimensions of $E(\nu_{\omega},d)\subset\overline E(\nu_{\omega},d)$ was already obtained in the proof of proposition~\ref{lower bound-1}, while the upper bound was obtained in the previous proposition for $d\in [\mathcal{T}'(+\infty),\mathcal{T}'(t_0-)]$, and it follows from the multifractal formalism for $d\in [\mathcal{T}'(t_0-),\mathcal{T}'(-\infty)]$, since $\tau_{\nu_\omega}^*\le\mathcal T^*$.
\end{proof}

%

\end{document}